\documentclass{scrartcl}
\usepackage{amsfonts}
\usepackage{amsmath}
\usepackage{amssymb}
\usepackage{hyperref}
\usepackage{graphicx, subfigure}
\newcommand{\bbbn}{{\mathbb N}}
\newcommand{\bbbr}{{\mathbb R}}
\newcommand{\Idx}{{\mathcal I}}
\newcommand{\Jdx}{{\mathcal J}}
\newcommand{\ctI}{{\mathcal T}_{\Idx}}
\newcommand{\ctJ}{{\mathcal T}_{\Jdx}}
\newcommand{\lfI}{{\mathcal L}_{\Idx}}
\newcommand{\ctIJ}{{\mathcal T}_{\Idx\times\Jdx}}
\newcommand{\lfIJ}{{\mathcal L}_{\Idx\times\Jdx}}
\newcommand{\lfaIJ}{{\mathcal L}^+_{\Idx\times\Jdx}}
\newcommand{\lfiIJ}{{\mathcal L}^-_{\Idx\times\Jdx}}
\newcommand{\sons}{\mathop{\operatorname{sons}}\nolimits}
\newcommand{\supp}{\mathop{\operatorname{supp}}\nolimits}
\newcommand{\diam}{\mathop{\operatorname{diam}}\nolimits}
\newcommand{\dist}{\mathop{\operatorname{dist}}\nolimits}
\allowdisplaybreaks
\newtheorem{theorem}{Theorem}
\newtheorem{lemma}[theorem]{Lemma}
\newtheorem{definition}[theorem]{Definition}
\newtheorem{remark}[theorem]{Remark}
\newenvironment{proof}{\emph{Proof.}}{\strut\hfill $\Box$\medskip}
\newcommand{\qed}{}
\title{Approximation of integral operators by Green quadrature
       and nested cross approximation.}
\author{Steffen B\"orm \and Sven Christophersen}
\date{\today}

\begin{document}

\maketitle

\begin{abstract}
\noindent
We present a fast algorithm that constructs a data-sparse approximation
of matrices arising in the context of integral equation methods for
elliptic partial differential equations.

The new algorithm uses Green's representation formula in combination
with quadrature to obtain a first approximation of the kernel function,
and then applies nested cross approximation to obtain a more efficient
representation.

The resulting ${\mathcal H}^2$-matrix representation requires
${\mathcal O}(n k)$ units of storage for an $n\times n$ matrix, where
$k$ depends on the prescribed accuracy.
\end{abstract}

MSC: 65N38, 65N80, 65D30, 45B05.

Keywords: Boundary element method, hierarchical matrix,
          Green's function, quadrature, cross approximation.

\thanks{\noindent We gratefully acknowledge that part of this research
  was supported by the Deutsche Forschungsgemeinschaft in the context of
  project BO~3289/2-1.}

\section{Introduction}

We consider integral equations of the form
\begin{align*}
  \int_\Omega g(x,y) u(y) \,dy &= f(x) &
  &\text{ for almost all } x\in\Omega.
\end{align*}
In order to solve these equations numerically, we choose
a trial space $\mathcal{U}_h$ and a test space $\mathcal{V}_h$ and
look for the Galerkin approximation $u_h\in\mathcal{U}_h$ satisfying
the variational equation
\begin{align*}
  \int_\Omega v_h(x) \int_\Omega g(x,y) u_h(y) \,dy\,dx
  &= \int_\Omega v_h(x) f(x) \,dx &
  &\text{ for all } v_h\in\mathcal{V}_h.
\end{align*}
If we fix bases $(\psi_j)_{j\in\Jdx}$ of $\mathcal{U}_h$ and
$(\varphi_i)_{i\in\Idx}$ of $\mathcal{V}_h$, the variational equation translates
into a linear system of equations
\begin{equation*}
  G \hat u = \hat f
\end{equation*}
with a matrix $G\in\bbbr^{\Idx\times\Jdx}$ given by
\begin{align}
  g_{ij} &= \int_\Omega \varphi_i(x) \int_\Omega g(x,y) \psi_j(y)
               \,dy\,dx &
  &\text{ for all } i\in\Idx,\ j\in\Jdx.\label{eq:matrix}
\end{align}
The matrix $G$ is typically non-sparse.
For standard applications in the field of elliptic partial differential
equations, we even have $g_{ij}\neq 0$ for \emph{all} $i\in\Idx$, $j\in\Jdx$.

Most techniques proposed to handle matrices of this type fall into one of
two categories:
\emph{kernel-based approximations} replace $g$ by a degenerate
approximation $\tilde g$ that can be treated efficiently,
while \emph{matrix-based approximations} work directly with
the matrix entries.

The popular multipole method \cite{RO85,GRRO87} relies originally
on a special expansion of the kernel function, the panel clustering
method \cite{HANO89} uses the more general Taylor expansion
or interpolation \cite{GI01,BOHA02a}, while ``multipole methods
without multipoles'' frequently rely on ``replacement sources'' located
around the domain selected for approximation \cite{AN92,BIYIZO04}.
Wavelet methods \cite{DAPRSC94b,DASC99,HASC06} implicitly use an
approximation of the kernel function that leads to a sparsification
of the matrix due to the vanishing-moment property of wavelet bases,
therefore we can also consider them as kernel-based approximations.

Matrix-based approximations, on the other hand, typically evaluate
a small number of matrix entries $g_{ij}$ and use these to construct
an approximation.
The cross-approximation approach \cite{TY96,GOTYZA97,TY99} computes
a small number of ``crosses'' consisting each of one row and one column
of submatrices that lead to low-rank approximations.
Combining this technique with a pivoting strategy and an error estimator
leads to the well-known adaptive cross approximation method
\cite{BE00a,BERJ01,BEGR06,MAMITI08,TAHERI11,BEVE12}.

Both kernel- and matrix-based approximations have advantages and
disadvantages.
Kernel-based approximations can typically be rigorously proven to
converge at a certain rate, and they do not depend on the choice
of basis functions or the mesh, but they are frequently less efficient
than matrix-based approximations.
Matrix-based approximations typically lead to very high compression
rates and can be used as black-box methods, but error estimates
currently depend either on computationally unfeasible pivoting
strategies (e.g., computing submatrices of maximal volume) or on
heuristics based on currently unproven stability assumptions.

\emph{Hybrid methods} try to combine kernel- and matrix-based
techniques in order to gain all the advantages and avoid most of the
disadvantages.
An example is the hybrid cross approximation technique
\cite{BOGR04} that applies cross approximation to a small submatrix
resulting from interpolation, thus avoiding the requirement of
possibly unreliable error estimators.
Another example is the kernel-independent multipole method
\cite{BIYIZO04} that uses replacement sources and solves a
regularized linear system to obtain an approximation.

The new algorithm we are presenting in this paper falls into
the hybrid category:
in a first step, an analytical scheme is used to obtain
a kernel approximation that leads to factorized approximation of
suitably-chosen matrix blocks.
In a second step, this approximation is compressed further by
applying a cross approximation method to certain factors appearing
in the first step, allowing us to improve the efficiency
significantly and to obtain an algebraic interpolation operator
that can be used to compute the final matrix approximation very
rapidly.

For the first step, we rely on the relatively recent concept
of \emph{quadrature-based approximations} \cite{BOGO12} that
can be applied to kernel functions resulting from typical
boundary integral formulations and takes advantage of
Green's representation formula in order to reduce the number
of terms.
Compared to standard techniques using Taylor expansion or
polynomial interpolation that require ${\mathcal O}(m^d)$
terms to obtain an $m$-th order approximation in $d$-dimensional
space, the quadrature-based approach requires only ${\mathcal O}(m^{d-1})$
terms and therefore has the same asymptotic complexity as the
original multipole method.
While the original article \cite{BOGO12} relies on the Leibniz
formula to derive an error estimate for the two-dimensional case,
we present a new proof that takes advantage of polynomial
best-approximation properties of the quadrature scheme in order
to obtain a more general result.
We consider the Laplace equation as a model problem, which leads
to the kernel function
\begin{equation*}
  g(x,y) = \frac{1}{4\pi \|x-y\|_2}
\end{equation*}
on a domain or submanifold $\Omega\subseteq\bbbr^3$, but we point
out that our approach carries over to other kernel functions
connected to representation equations, e.g., it is applicable to the
low-frequency Helmholtz equation (cf. \cite[eq. (2.1.5)]{HSWE08}),
the Lam\'e equation (cf. \cite[eq. (2.2.4)]{HSWE08}, the Stokes equation
(cf. \cite[eq. (2.3.8)]{HSWE08}), or the biharmonic equation
(cf. \cite[eq. (2.4.6)]{HSWE08}).

\section[H2-matrices]{${\mathcal H}^2$-matrices}

Since we are not able to approximate the entire matrix at once,
we consider submatrices.
Hierarchical matrix methods \cite{HA99,HAKH00,GRHA02} choose
these submatrices based on a hierarchy of subsets.

%
%
\begin{definition}[Cluster tree]
Let $\Idx$ denote a finite index set.
Let ${\mathcal T}$ be a labeled tree, and denote the label of
a node $t\in{\mathcal T}$ by $\hat t$.
We call ${\mathcal T}$ a \emph{cluster tree} for $\Idx$ if
\begin{itemize}
  \item the root $r\in{\mathcal T}$ has the label $\hat r=\Idx$,
  \item any node $t\in{\mathcal T}$ with $\sons(t)\neq\emptyset$
     satisfies $\hat t = \bigcup_{t'\in\sons(t)} \hat t'$, and
  \item any two different sons $t_1,t_2\in\sons(t)$ of $t\in{\mathcal T}$
     satisfy $\hat t_1\cap\hat t_2 = \emptyset$.
\end{itemize}
The nodes of a cluster tree are called \emph{clusters}.
A cluster tree for an index set $\Idx$ is denoted by $\ctI$, the
corresponding set of leaves by $\lfI := \{ t\in\ctI\ :\ \sons(t)=\emptyset \}$.
\end{definition}

Submatrices of a matrix $G\in\bbbr^{\Idx\times\Jdx}$ are represented
by pairs of clusters chosen from two cluster trees $\ctI$ and
$\ctJ$ for the index sets $\Idx$ and $\Jdx$, respectively.
In order to find suitable submatrices efficienly, these pairs
are also organized in a tree structure.

%
%
\begin{definition}[Block tree]
Let $\ctI$ and $\ctJ$ be cluster trees for index sets $\Idx$
and $\Jdx$ with roots $r_\Idx$ and $r_\Jdx$.
Let ${\mathcal T}$ be a labeled tree, and denote the label of
a node $b\in{\mathcal T}$ by $\hat b$.
We call ${\mathcal T}$ a \emph{block tree} for $\ctI$ and $\ctJ$ if
\begin{itemize}
  \item for each $b\in{\mathcal T}$, there are $t\in\ctI$ and
     $s\in\ctJ$ with $b=(t,s)$ and $\hat b=\hat t\times\hat s$,
  \item the root $r\in{\mathcal T}$ satisfies $r=(r_\Idx,r_\Jdx)$,
  \item for each $b=(t,s)\in{\mathcal T}$ with $\sons(b)\neq\emptyset$,
     we have
     \begin{equation*}
       \sons(b) = \begin{cases}
         \sons(t)\times\{s\}
         &\text{ if } \sons(t)\neq\emptyset
          \text{ and } \sons(s)=\emptyset,\\
         \{t\}\times\sons(s)
         &\text{ if } \sons(t)=\emptyset
          \text{ and } \sons(s)\neq\emptyset,\\
         \sons(t)\times\sons(s)
         &\text{ otherwise}.
       \end{cases}
     \end{equation*}
\end{itemize}
The nodes of a block tree are called \emph{blocks}.
A block tree for cluster trees $\ctI$ and $\ctJ$ is denoted by
$\ctIJ$, the corresponding set of leaves by $\lfIJ$.
\end{definition}

In the following we assume that index sets $\Idx$ and $\Jdx$
with corresponding cluster trees $\ctI$ and $\ctJ$ and a block
tree $\ctIJ$ are given.

It is easy to see that $\ctIJ$ is itself a cluster tree for the
Cartesian product index set $\Idx\times\Jdx$.
A simple induction shows that for any cluster tree, the leaves'
labels form a disjoint partition of the corresponding index
set.
In particular, the leaves of the block tree $\ctIJ$ correspond
to a disjoint partition
\begin{equation*}
  \{ \hat t\times\hat s\ :\ b=(t,s)\in\lfIJ \}
\end{equation*}
of the product index set $\Idx\times\Jdx$ corresponding to the
matrix.
This property allows us to define an approximation of a matrix
$G\in\bbbr^{\Idx\times\Jdx}$ by choosing approximations for all
submatrices $G|_{\hat t\times\hat s}$ corresponding to leaf
blocks $b=(t,s)\in\lfIJ$.

Since we cannot approximate all blocks equally well, we use
an \emph{admissibility condition}
\begin{equation}\label{eq:admissibility}
  \operatorname{adm} : \ctI\times\ctJ \to
         \{ \operatorname{true}, \operatorname{false} \}
\end{equation}
that indicates which blocks can be approximated.
A block $(t,s)$ is called \emph{admissible} if
$\operatorname{adm}(t,s)=\operatorname{true}$ holds.

%
%
\begin{definition}[Admissible block tree]
The block tree $\ctIJ$ is called \emph{admissible} if
\begin{align*}
  \operatorname{adm}(t,s) &\vee
    \sons(t)=\emptyset \vee \sons(s)=\emptyset &
  &\text{ for all leaves } b=(t,s)\in\lfIJ.
\end{align*}
It is called \emph{strictly admissible} if
\begin{align*}
  \operatorname{adm}(t,s) &\vee
    (\sons(t)=\emptyset \wedge \sons(s)=\emptyset) &
  &\text{ for all leaves } b=(t,s)\in\lfIJ.
\end{align*}
\end{definition}

Given cluster trees $\ctI$ and $\ctJ$ and an admissibility
condition, a minimal admissible (or strictly admissible) block
tree can be constructed by starting with the root pair
$(r_\Idx,r_\Jdx)$ and checking whether it is admissible.
If it is, we are done.
Otherwise, we recursively check its sons and further descendants
\cite{HAKH00}.

Admissible leaves $b=(t,s)\in\lfIJ$ correspond to submatrices
$G|_{\hat t\times\hat s}$ that can be approximated, while inadmissible
leaves correspond to submatrices that have to be stored directly.
To distinguish between both cases, we let
\begin{align*}
  \lfaIJ &:= \{ b=(t,s)\in\lfIJ\ :\ \operatorname{adm}(t,s)
                 = \operatorname{true} \}, &
  \lfiIJ &:= \lfIJ \setminus \lfaIJ.
\end{align*}
Defining an approximation of $G$ means defining approximations
for all $G|_{\hat t\times\hat s}$ with $b=(t,s)\in\lfaIJ$.

%
%
\begin{definition}[Hierarchical matrix]
A matrix $G\in\bbbr^{\Idx\times\Jdx}$ is called a
\emph{hierarchical matrix} with local rank $k\in\bbbn$, if for
each $b=(t,s)\in\lfaIJ$ we can find $A_b\in\bbbr^{\hat t\times k}$
and $B_b\in\bbbr^{\hat s\times k}$ such that
\begin{equation*}
  G|_{\hat t\times\hat s} = A_b B_b^*,
\end{equation*}
where $B_b^*\in\bbbr^{k\times\hat s}$ denotes the transposed of the
matrix $B_b$.
\end{definition}

In typical applications, representing all admissible submatrices
by the factors $A_b$ and $B_b$ reduces the storage requirements
for a hierarchical matrix to ${\mathcal O}(n k \log(n))$,
where $n:=\max\{\#\Idx,\#\Jdx\}$ \cite{GRHA02}.

The logarithmic factor can be avoided by refining the
representation:
we choose sets of basis vectors for all clusters $t\in\ctI$
and $s\in\ctJ$ and represent the admissible blocks in terms
of these basis vectors.

%
%
\begin{definition}[Cluster basis]
A family $(V_t)_{t\in\ctI}$ of matrices $V_t\in\bbbr^{\hat t\times k}$
is called a \emph{cluster basis} of rank $k$.
\end{definition}

%
%
\begin{definition}[Uniform hierarchical matrix]
A matrix $G\in\bbbr^{\Idx\times\Jdx}$ is called a
\emph{uniform hierarchical matrix} for cluster bases
$(V_t)_{t\in\ctI}$ and $(W_s)_{s\in\ctJ}$, if for each
$b=(t,s)\in\lfaIJ$ we can find $S_b\in\bbbr^{k\times k}$
such that
\begin{equation*}
  G|_{\hat t\times\hat s} = V_t S_b W_s^*.
\end{equation*}
The matrices $S_b$ are called \emph{coupling matrices}.
\end{definition}

Although a uniform hierarchical matrix requires only $k^2$
units of storage \emph{per block}, leading to total storage
requirements of ${\mathcal O}(n k)$, the cluster bases still
need ${\mathcal O}(n k \log(n))$ units of storage.
In order to obtain \emph{linear} complexity, we assume that
the cluster bases match the hierarchical structure of the
cluster trees, i.e., that the bases of father clusters can
be expressed in terms of the bases of the sons.

%
%
\begin{definition}[Nested cluster basis]
A cluster basis $(V_t)_{t\in\ctI}$ is called \emph{nested},
if for each $t\in\ctI$ and each $t'\in\sons(t)$ there is
a matrix $E_{t'}\in\bbbr^{k\times k}$ such that
\begin{equation*}
  V_t|_{\hat t'\times k} = V_{t'} E_{t'}.
\end{equation*}
The matrices $E_{t'}$ are called \emph{transfer matrices}.
\end{definition}

%
%
\begin{definition}[${\mathcal H}^2$-matrix]
Let $G\in\bbbr^{\Idx\times\Jdx}$ be a uniform hierarchical matrix
for cluster bases $(V_t)_{t\in\ctI}$ and $(W_s)_{s\in\ctJ}$.
If the cluster bases are nested, $G$ is called an
\emph{${\mathcal H}^2$-matrix}.
\end{definition}

In typical applications, representing all admissible submatrices
by the coupling matrices and the cluster bases by the transfer matrices
reduces the storage requirements for an ${\mathcal H}^2$-matrix
to ${\mathcal O}(n k)$.

The remainder of this article is dedicated to the task of finding an
efficient algorithm for constructing ${\mathcal H}^2$-matrix
approximations of matrices corresponding to the Galerkin
discretization of integral operators.

\section{Representation formula and quadrature}

Applying cross approximation directly to matrix blocks would lead
either to a very high computational complexity (if full pivoting
is used) or to a potentially unreliable method (if a heuristic
privoting strategy with a heuristic error estimator is employed).

Since we are interested in constructing a method that is both
fast and reliable, we follow the approach of \emph{hybrid cross
approximation} \cite{BOGR04}:
in a first step, an analytic technique is used to obtain a degenerate
approximation of the kernel function.
In a second step, an algebraic technique is used to reduce the
storage requirements of the approximation obtained in the first step,
in our case by a reliable cross approximation constructed with
full pivoting.
Applying a modification similar to \cite{BEVE12}, this approach
leads to an efficient $\mathcal{H}^2$-matrix approximation.

We follow the approach described in \cite{BOGO12} for the Laplace
equation, since it offers optimal-order ranks and is very robust:
let $d\in\{2,3\}$, let $\omega\subseteq\bbbr^d$ be a Lipschitz domain,
and let $u:\overline{\omega}\to\bbbr$ be harmonic in $\omega$.
Green's representation formula (cf., e.g., \cite[Theorem~2.2.2]{HA92})
states
\begin{align*}
  u(x) &= \int_{\partial\omega} g(x,z)
                             \frac{\partial u}{\partial n}(z) \,dz
        - \int_{\partial\omega} \frac{\partial g}{\partial n(z)}(x,z)
                            u(z) \,dz &
  &\text{ for all } x\in\omega,
\end{align*}
where
\begin{equation*}
  g(x,y) = \begin{cases}
    -\frac{1}{2\pi} \log \|x-y\|_2 &\text{ if } d=2,\\
    \frac{1}{4\pi} \frac{1}{\|x-y\|_2} &\text{ if } d=3
  \end{cases}
\end{equation*}
denotes a fundamental solution of the negative Laplace operator $-\Delta$.

For any $y\not\in\bar\omega$, the function $u(x) = g(x,y)$ is harmonic,
so we can apply the formula to obtain
\begin{align}
  g(x,y) &= \int_{\partial\omega} g(x,z)
               \frac{\partial g}{\partial n(z)}(z,y) \,dz
          - \int_{\partial\omega} \frac{\partial g}{\partial n(z)}(x,z)
               g(z,y) \,dz\label{eq:green_g}\\
  &\qquad\text{ for all } x\in\omega,\ y\not\in\bar\omega.\notag
\end{align}
On the right-hand side, the variables $x$ and $y$ no longer appear
together as arguments of $g$ or $\partial g/\partial n$, the integrands
are tensor products.

If $x$ and $y$ are sufficiently far from the boundary $\partial\omega$,
the integrands are smooth, so we can approximate the integrals by
an exponentially convergent quadrature rule.
Denoting its weights by $(w_\nu)_{\nu\in K}$ and its quadrature points
by $(z_\nu)_{\nu\in K}$, we find the approximation
\begin{align}\label{eq:green_g_apx1}
  g(x,y) &\approx \sum_{\nu\in K} w_\nu g(x,z_\nu)
                \frac{\partial g}{\partial n(z_\nu)}(z_\nu,y)
              - w_\nu \frac{\partial g}{\partial n(z_\nu)}(x,z_\nu)
                g(z_\nu,y)\\
  &\qquad \text{ for all } x\in\omega,\ y\not\in\bar\omega.\notag
\end{align}
Since this is a degenerate approximation of the kernel function,
discretizing the corresponding integral operator directly leads
to a hierarchical matrix.

In order to ensure uniform exponential convergence of the approximation,
we have to choose a suitable admissibility condition that ensures
that $x$ and $y$ are sufficiently far from the boundary $\partial\omega$.

A simple approach relies on bounding boxes:
given a cluster $t\in\ctI$, we assume that there is an axis-parallel
box
\begin{equation*}
  {\mathcal B}_t = [a_{t,1},b_{t,1}]\times\ldots\times[a_{t,d},b_{t,d}]
\end{equation*}
containing the supports of all basis functions corresponding to
indices in $\hat t$, i.e., such that
\begin{align*}
  \supp \varphi_i &\subseteq {\mathcal B}_t &
  &\text{ for all } i\in\hat t.
\end{align*}
These bounding boxes can be constructed efficiently by
a recursive algorithm \cite{BOGRHA03a}.

In order to be able to apply the quadrature approximation to a
cluster $t\in\ctI$, we have to ensure that $x$ and $y$ are at
a ``safe distance'' from $\partial\omega$.
In view of the error estimates presented in \cite{BOGO12}, we
denote the \emph{farfield} of $t$ by
\begin{equation}\label{eq:farfield}
  {\mathcal F}_t
  := \{ y\in\bbbr^d\ :\ \diam_\infty({\mathcal B}_t)\leq
                          \dist_\infty({\mathcal B}_t,y) \},
\end{equation}
where diameter and distance with respect to the maximum norm are given by
\begin{align*}
  \diam_\infty({\mathcal B}_t)
  &:= \max\{ \|x-y\|_\infty\ :\ x,y\in {\mathcal B}_t \},\\
  \dist_\infty({\mathcal B}_t, y)
  &:= \min\{ \|x-y\|_\infty\ :\ x\in {\mathcal B}_t \}.
\end{align*}
We are looking for an approximation that yields a sufficiently small error
for all $x\in{\mathcal B}_t$ and all $y\in{\mathcal F}_t$.
We apply Green's representation formula to the domain $\omega_t$
given by
\begin{align*}
  \delta_t &:= \diam_\infty({\mathcal B}_t)/2, &
  \omega_t &:= [a_{t,1}-\delta_t, b_{t,1}+\delta_t]\times\ldots
               \times [a_{t,d}-\delta_t, b_{t,d}+\delta_t].
\end{align*}
It is convenient to represent $\omega_t$ by means of a reference
cube $[-1,1]^d$ using the affine mapping
\begin{align}\label{eq:Phi_t}
  \Phi_t : [-1,1]^d &\to \omega_t, &
           \hat x &\mapsto \frac{b+a}{2} + \frac{1}{2}
              \begin{pmatrix}
                b_{t,1}-a_{t,1}+2 \delta_t & & \\
                & \ddots & \\
                & & b_{t,d}-a_{t,d}+2 \delta_t
              \end{pmatrix} \hat x,
\end{align}
and this directly leads to affine parametrizations
\begin{subequations}\label{eq:parametrization}
\begin{align}
  \gamma_{2\iota-1}(\hat z)
  &:= \Phi_t(\hat z_1,\ldots,\hat z_{\iota-1},-1,\hat z_\iota,
              \ldots,\hat z_{d-1}),\\
  \gamma_{2\iota}(\hat z)
  &:= \Phi_t(\hat z_1,\ldots,\hat z_{\iota-1},+1,\hat z_\iota,
              \ldots,\hat z_{d-1})\\
  &\qquad\text{for all } \iota\in\{1,\ldots,d\},\ \hat z\in Q,
            \notag
\end{align}
\end{subequations}
of the boundary $\partial\omega_t$, where
\begin{equation*}
  Q := [-1,1]^{d-1}
\end{equation*}
is the parameter domain for one side of the boundary, such that
\begin{align*}
  \partial\omega_t &= \bigcup_{\iota=1}^{2d} \gamma_\iota(Q), &
  \int_{\partial\omega_t} f(z) \,dz &= \sum_{\iota=1}^{2d}
       \int_Q \sqrt{\det D\gamma_\iota^* D\gamma_\iota}
              f(\gamma_\iota(\hat z)) \,d\hat z.
\end{align*}
We approximate the integrals on the right-hand side by a tensor
quadrature formula:
let $m\in\bbbn$, let $\xi_1,\ldots,\xi_m\in[-1,1]$ denote the
points and $w_1,\ldots,w_m\in\bbbr$ the weights of the one-dimensional
$m$-point Gauss quadrature formula for the reference interval $[-1,1]$.
If we define
\begin{align*}
  \hat z_\mu &:= (\xi_{\mu_1},\ldots,\xi_{\mu_{d-1}}), &
  \hat w_\mu &:= w_{\mu_1} \cdots w_{\mu_{d-1}} &
  &\text{ for all } \mu\in M:=\{1,\ldots,m\}^{d-1},
\end{align*}
we obtain the tensor quadrature formula
\begin{equation*}
  \int_Q \hat f(\hat z) \,d\hat z
  \approx \sum_{\mu\in M} \hat w_\mu \hat f(\hat z_\mu).
\end{equation*}
Applying this result to all surfaces of $\partial\omega_t$ yields
\begin{equation*}
  \int_{\partial\omega} f(z) \,dz
  \approx \sum_{\iota=1}^{2d} \sum_{\mu\in M}
            \hat w_\mu \sqrt{\det D\gamma_\iota^* D\gamma_\iota}
            f(\gamma_\iota(\hat z_\mu))
  = \sum_{(\iota,\mu)\in K} w_{\iota\mu}
            f(z_{\iota\mu}),
\end{equation*}
where we define
\begin{align*}
  K &:= \{1,\ldots,2d\}\times M, &
  w_{\iota\mu} &:= \hat w_\mu \sqrt{\det D\gamma_\iota^* D\gamma_\iota}, &
  z_{\iota\mu} &:= \gamma_\iota(\hat z_\mu).
\end{align*}
Using this quadrature formula in (\ref{eq:green_g_apx1}) yields
\begin{align}\label{eq:green_g_apx}
  \tilde g_t(x,y)
  &:= \sum_{(\iota,\mu)\in K}
        w_{\iota\mu}
        g(x, z_{\iota\mu})
        \frac{\partial g}{\partial n_\iota}(z_{\iota\mu}, y)
        - w_{\iota\mu}
        \frac{\partial g}{\partial n_\iota}(x, z_{\iota\mu})
        g(z_{\iota\mu}, y)\\
  &= \sum_{(\iota,\mu)\in K}
        w_{\iota\mu}
        g(x, z_{\iota\mu})
        \frac{\partial g}{\partial n_\iota}(z_{\iota\mu}, y)
        - w_{\iota\mu} \delta_t
        \frac{\partial g}{\partial n_\iota}(x, z_{\iota\mu})
        \frac{1}{\delta_t} g(z_{\iota\mu}, y)\notag\\
  &\qquad\text{ for all } x\in{\mathcal B}_t,\ y\in{\mathcal F}_t,\notag
\end{align}
where $n_\iota$ denotes the outer normal vector of the face $\gamma_\iota(Q)$
of $\omega_t$.
The additional scaling factors in the second row have been added to
make the estimate of Lemma~\ref{le:hybrid_scaling} more elegant by
compensating for the different singularity orders of the integrands.

We use the admissibility condition (cf. \ref{eq:admissibility}) given by
the relative distance of clusters:
a block $b=(t,s)$ is admissible if $\mathcal{B}_s$ is in the farfield
of $t$, i.e., if ${\mathcal B}_s\subseteq{\mathcal F}_t$ holds.
Given such an admissible block $b=(t,s)$, replacing the kernel function $g$
by $\tilde g_t$ in (\ref{eq:matrix}) leads to the low-rank approximation
\begin{align}\label{eq:green_ab}
  G|_{\hat t\times\hat s}
  &\approx A_t B_{ts}^*, &
  A_t &= \begin{pmatrix} A_{t+} & A_{t-} \end{pmatrix}, &
  B_{ts} &= \begin{pmatrix} B_{ts+} & B_{ts-} \end{pmatrix},
\end{align}
where the low-rank factors $A_{t+},A_{t-}\in\bbbr^{\hat t\times K}$ and
$B_{ts+},B_{ts-}\in\bbbr^{\hat s\times K}$ are given by
\begin{subequations}\label{eq:ab_def}
\begin{gather}
  a_{t+,i\nu} := \sqrt{w_\nu}
               \int_\Omega \varphi_i(x) g(x,z_\nu)\,dx,\quad
  b_{ts+,j\nu} := \sqrt{w_\nu} \int_\Omega \psi_j(y)
                      \frac{\partial g}{\partial n_\iota}(z_\nu,y) \,dy,\\
  a_{t-,i\nu} := \delta_t \sqrt{w_\nu}
               \int_\Omega \varphi_i(x)
                     \frac{\partial g}{\partial n_\iota}(x,z_\nu)\,dx,\quad
  b_{ts-,j\nu} := -\frac{\sqrt{w_\nu}}{\delta_t}
               \int_\Omega \psi_j(y)
                      g(z_\nu,y) \,dy\\
  \qquad\qquad\text{ for all } \nu=(\iota,\mu)\in K,
               \ i\in\hat t,\ j\in\hat s.\notag
\end{gather}
\end{subequations}
It is important to note that $A_t$ depends only on $t$, but not on $s$.
This property allows us to extend our construction to obtain
${\mathcal H}^2$-matrices in later sections.

%
%
\begin{remark}[Complexity]
We have $\#K = 2dm^{d-1}$ by definition, therefore $A_t B_{ts}^*$ is
an approximation of rank $4dm^{d-1}$.
Standard complexity estimates for hierarchical matrices
(cf. \cite[Lemma 2.4]{GRHA02}) allow us to conclude that the
resulting approximation requires ${\mathcal O}(n m^{d-1} \log n)$
units of storage.
If we assume that the entries of $A_t$ and $B_{ts}$ and the nearfield
matrices are computed by constant-order quadrature, the hierarchical
matrix representation can be constructed in ${\mathcal O}(n m^{d-1} \log n)$
operations.
\end{remark}

\section{Convergence of the quadrature approximation}

The error analysis in \cite{BOGO12} relies on Leibniz' formula to
obtain estimates of the derivatives of the integrand in (\ref{eq:green_g}).
Here we present an alternative proof that handles the integrand's product
directly.

The fundamental idea is the following:
since the one-dimensional formula yields the exact integral for
polynomials of degree $2m-1$, the tensor formula yields the exact
integral for tensor products of polynomials of this degree, i.e., we have
\begin{align}\label{eq:gauss_exact}
  \int_Q \hat p(\hat z) \,d\hat z
  &= \sum_{\nu\in M} \hat w_\nu \hat p(\hat z_\nu) &
  &\text{ for all } \hat p\in {\mathcal Q}_{2m-1},
\end{align}
where ${\mathcal Q}_{2m-1}$ denotes the space of $(d-1)$-dimensional
tensor products of polynomials of degree $2m-1$.

Applying (\ref{eq:gauss_exact}) to the constant polynomial $\hat p=1$
and taking advantage of the fact that Gauss weights are non-negative,
we obtain
\begin{equation}\label{eq:gauss_stable}
  \sum_{\nu\in M} |w_\nu| = \sum_{\nu\in M} w_\nu = 2^{d-1}.
\end{equation}
Combining (\ref{eq:gauss_exact}) and (\ref{eq:gauss_stable}) leads
to the following well-known best-approximation estimate.

%
%
\begin{lemma}[Quadrature error]
\label{le:quadrature_single}
Let $\hat f\in C(Q)$.
We have
\begin{align*}
  \left| \int_Q \hat f(\hat z) \,d\hat z
         - \sum_{\nu\in M} \hat w_\nu \hat f(\hat z_\nu) \right|
  &\leq 2^d \|\hat f - \hat p\|_{\infty,Q} &
  &\text{ for all } \hat p\in {\mathcal Q}_{2m-1}.
\end{align*}
\end{lemma}
\begin{proof}
Let $\hat p\in{\mathcal Q}_{2m-1}$.
Due to (\ref{eq:gauss_exact}) and (\ref{eq:gauss_stable}), we have
\begin{align*}
  \left| \int_Q \hat f(\hat z) \,d\hat z
         - \sum_{\nu\in M} \hat w_\nu \hat f(\hat z_\nu) \right|
  &= \left| \int_Q (\hat f - \hat p)(\hat z) \,d\hat z
         - \sum_{\nu\in M} \hat w_\nu (\hat f - \hat p)(\hat z_\nu) \right|\\
  &\leq \left| \int_Q (\hat f - \hat p)(\hat z) \,d\hat z \right|
         + \left| \sum_{\nu\in M} \hat w_\nu
                      (\hat f - \hat p)(\hat z_\nu) \right|\\
  &\leq \int_Q \|\hat f - \hat p\|_{\infty,Q} \,d\hat z
        + \sum_{\nu\in M} |\hat w_\nu|\, \|\hat f - \hat p\|_{\infty,Q}\\
  &= 2^{d-1} \|\hat f - \hat p\|_{\infty,Q}
        + 2^{d-1} \|\hat f - \hat p\|_{\infty,Q}\\
  &= 2^d \|\hat f - \hat p\|_{\infty,Q}.
\end{align*}
\qed
\end{proof}

We are interested in approximating the integrals appearing in
(\ref{eq:green_g}), where the integrands are products.
Fortunately, Lemma~\ref{le:quadrature_single} can be easily
extended to products.

%
%
\begin{lemma}[Products]
\label{le:quadrature_product}
Let $\hat f,\hat g\in C(Q)$.
We have
\begin{align*}
  \left| \int_Q \hat f(\hat z) \hat g(\hat z) \,d\hat z \right.
       &\left.  - \sum_{\nu\in M} \hat w_\nu \hat f(\hat z_\nu)
            \hat g(\hat z_\nu) \right|\\
  &\leq 2^d ( \|\hat f - \hat p\|_{\infty,Q} \|\hat g\|_{\infty,Q}
              + \|\hat f\|_{\infty,Q} \|\hat g - \hat q\|_{\infty,Q}\\
  &\qquad  + \|\hat f - \hat p\|_{\infty,Q}
               \|\hat g - \hat q\|_{\infty,Q})
   \qquad\text{ for all } \hat p\in {\mathcal Q}_m, \hat q\in {\mathcal Q}_{m-1}.
\end{align*}
\end{lemma}
\begin{proof}
Let $\hat p\in{\mathcal Q}_m$ and $\hat q\in{\mathcal Q}_{m-1}$.
Then we have $\hat p\hat q\in{\mathcal Q}_{2m-1}$ and
Lemma~\ref{le:quadrature_single} yields
\begin{equation*}
  \left| \int_Q \hat f(\hat z) \hat g(\hat z) \,d\hat z
         - \sum_{\nu\in M} \hat w_\nu \hat f(\hat z_\nu) \hat g(\hat z_\nu)
  \right| \leq 2^d \|\hat f\hat g - \hat p\hat q\|_{\infty,Q}.
\end{equation*}
Observing
\begin{align*}
  \|\hat f \hat g - \hat p \hat q\|_{\infty,Q}
  &= \| (\hat f - \hat p) \hat g + \hat p (\hat g - \hat q) \|_{\infty,Q}\\
  &= \| (\hat f - \hat p) \hat g
        + \hat f (\hat g - \hat q)
        - (\hat f - \hat p) (\hat g - \hat q) \|_{\infty,Q}\\
  &\leq \|\hat f - \hat p\|_{\infty,Q} \|\hat g\|_{\infty,Q}
        + \|\hat f\|_{\infty,Q} \|\hat g - \hat q\|_{\infty,Q}\\
  &\qquad + \|\hat f - \hat p\|_{\infty,Q} \|\hat g - \hat q\|_{\infty,Q}
\end{align*}
completes the proof.
\qed
\end{proof}

In our application, we want to use quadrature to approximate the integrals
\begin{align*}
  \int_{\partial\omega_t}
    & g(x,z) \frac{\partial g}{\partial n(z)}(z,y)
    - \frac{\partial g}{\partial n(z)}(x,z) g(z,y) \,dz\\
  &= \sum_{\iota=1}^{2d} \sqrt{\det D\gamma_\iota^* D\gamma_\iota} \int_Q
    g(x,\gamma_\iota(\hat z))
    \frac{\partial g}{\partial n_\iota}(\gamma_\iota(\hat z),y)
    - \frac{\partial g}{\partial n_\iota}(x,\gamma_\iota(\hat z))
    g(\gamma_\iota(\hat z),y) \,d\hat z\\
  &= \sum_{\iota=1}^{2d} \sqrt{\det D\gamma_\iota^* D\gamma_\iota} \int_Q
    \hat f_1(\hat z) \hat g_1(\hat z)
    - \hat g_2(\hat z) \hat f_2(\hat z) \,d\hat z,
\end{align*}
where $\hat f_1$, $\hat f_2$, $\hat g_1$ and $\hat g_2$ are given by
\begin{subequations}
\begin{align}\label{eq:fg_first}
  \hat f_1(\hat z)
  &:= g(x,\gamma_\iota(\hat z)), &
  \hat g_1(\hat z)
  &:= \frac{\partial g}{\partial n_\iota}(\gamma_\iota(\hat z),y),\\
  \hat f_2(\hat z)
  &:= g(\gamma_\iota(\hat z),y), &
  \hat g_2(\hat z)
  &:= \frac{\partial g}{\partial n_\iota}(x,\gamma_\iota(\hat z)).
\end{align}
\end{subequations}
Therefore we are looking for polynomial approximations of these
functions in order to apply Lemma~\ref{le:quadrature_product}.
We will look for $\hat p_1,\hat p_2\in{\mathcal Q}_m$ approximating
$\hat f_1,\hat f_2$ and for $\hat q_1,\hat q_2\in{\mathcal Q}_{m-1}$
approximating $\hat g_1,\hat g_2$.
This task can be solved easily using the framework developed in
\cite[Chapter 4]{BO10}, in particular the following result:

%
%
\begin{theorem}[Chebyshev interpolation]
\label{th:chebyshev}
Let ${\mathfrak I}_m^Q:C(Q)\to{\mathcal Q}_m$ denote the $m$-th order
tensor Chebyshev interpolation operator.
Let $f\in C^\infty(Q)$ with $C_f\in\bbbr_{\geq 0}$ and $\gamma_f\in\bbbr_{>0}$
such that
\begin{align}\label{eq:analyticity}
  \left\| \frac{\partial^n}{\partial z_\iota^n} f \right\|_{\infty,Q}
  &\leq \frac{C_f}{\gamma_f^n} n! &
  &\text{ for all } \iota\in\{1,\ldots,d\},\ n\in\bbbn_0
\end{align}
holds.
Then we have
\begin{equation*}
  \|f - {\mathfrak I}_m^Q[f]\|_{\infty,Q}
  \leq 2 d e C_f (\Lambda_m+1)^d
       \left( 1 + \frac{\diam_\infty(Q)}{\gamma_f} \right)
       (m + 1)
       \varrho\left(\frac{2 \gamma_f}{b_{t,\iota}-a_{t,\iota}} \right)^{-m},
\end{equation*}
where $\Lambda_m\leq m+1$ denotes the stability constant of
one-dimensional Chebyshev interpolation and
\begin{equation*}
  \varrho(r) := r + \sqrt{1+r^2} > r+1.
\end{equation*}
\end{theorem}
\begin{proof}
cf. \cite[Theorem~4.20]{BO10} in the isotropic case with $\sigma=1$.
\qed
\end{proof}

If we can satisfy the analyticity condition (\ref{eq:analyticity}),
this theorem provides us with a polynomial in ${\mathcal Q}_m$.
Since the functions $\hat f_1$, $\hat f_2$, $\hat g_1$ and $\hat g_2$
directly depend on the kernel function $g$, we cannot proceed without
taking the latter's properties into account.

In particular, we assume that $g$ is \emph{asymptotically smooth}
(cf. \cite{BRLU90,BR91,HA09}), i.e., that there are constants
$C_{\rm as},c_0\in\bbbr_{\geq 0}$, $\sigma\in\bbbn_0$ such that
\begin{align}\label{eq:asymptotically_smooth}
  | \partial^\nu_x \partial^\mu_y g(x,y)| 
  &\leq C_{\rm as} (\nu+\mu)!
        \frac{c_0^{|\nu|+|\mu|}}{\|x-y\|^{\sigma+|\nu|+|\mu|}} &
  &\begin{aligned}[t]
      &\text{for all } \nu,\mu\in\bbbn^d,\\
      &\quad x,y\in\bbbr^d\text{ with } x\neq y.
    \end{aligned}
\end{align}
The asymptotic smoothness of the fundamental solutions of the
Laplace operator $\Delta$ and other important kernel functions is
well-established, cf., e.g., \cite[Satz~E.1.4]{HA09}.

We have to be able to bound the right-hand side of the estimate
(\ref{eq:asymptotically_smooth}).
Since we have chosen the domain $\omega_t$ appropriately, we can
easily find the required estimate.

%
%
\begin{lemma}[Domain and parametrization]
\label{le:omega_t}
The domain $\omega_t$ satisfies the following estimates:
\begin{align}\label{eq:diam_dist}
  \diam_\infty(\omega_t) &= 4 \delta_t, &
  \dist_\infty({\mathcal B}_t, \partial\omega_t) &= \delta_t, &
  \dist_\infty(\partial\omega_t, {\mathcal F}_t) &\geq \delta_t, &
  |\partial\omega_t| &\leq 2 d (4 \delta_t)^{d-1}.
\end{align}
The parametrizations are bijective and satisfy
\begin{align}\label{eq:derivative}
  \|D \Phi_t\|_2 &= 2\delta_t, &
  \|D \gamma_\iota\|_2 &\leq 2 \delta_t.
\end{align}
\end{lemma}
\begin{proof}
cf. Appendix~\ref{se:appendix}.
\end{proof}

Combining these lower bounds for the distances between $x$ and
$\partial\omega_t$ and $y$ and $\partial\omega_t$, respectively,
with the estimate (\ref{eq:asymptotically_smooth}) allows us to
prove that the requirements of Theorem~\ref{th:chebyshev} are
fulfilled.

%
%
\begin{lemma}[Derivatives]
\label{le:derivatives}
Let $x\in{\mathcal B}_t$ and $y\in{\mathcal F}_t$.
Then we have
\begin{align*}
  |\partial^{\hat\nu} \hat f_1(\hat z)|,
  &\leq \frac{C_{\rm as} \hat\nu!}{\delta_t^\sigma} (2 c_0)^{|\hat\nu|}, &
  |\partial^{\hat\nu} \hat f_2(\hat z)|,
  &\leq \frac{C_{\rm as} \hat\nu!}{\delta_t^\sigma} (2 c_0)^{|\hat\nu|},\\
  |\partial^{\hat\nu} \hat g_1(\hat z)|,
  &\leq \frac{C_{\rm as} \hat\nu!}{\delta_t^{\sigma+1}} (2 c_0)^{|\hat\nu|}, &
  |\partial^{\hat\nu} \hat g_2(\hat z)|,
  &\leq \frac{C_{\rm as} \hat\nu!}{\delta_t^{\sigma+1}} (2 c_0)^{|\hat\nu|}
    \text{ for all } \hat z\in Q,\ \hat\nu\in\bbbn_0^{d-1}.
\end{align*}
\end{lemma}
\begin{proof}
cf. Appendix~\ref{se:appendix}
\end{proof}

Due to Lemma~\ref{le:derivatives}, the conditions of
Theorem~\ref{th:chebyshev} are fulfilled, so we can obtain the
polynomial approximations required by Lemma~\ref{le:quadrature_product}
and prove that the quadrature approximation converges exponentially.

%
%
\begin{theorem}[Quadrature error]
\label{th:quadrature}
There is a constant $C_{\rm gr}\in\bbbr_{\geq 0}$ depending only on
$C_{\rm as}$, $c_0$ and $d$ such that
\begin{align*}
  |g(x,y) - \tilde g_t(x,y)|
  &\leq \frac{C_{\rm gr}}{\delta_t^{2\sigma-d+2}}
        \left(\frac{2c_0}{2c_0+1}\right)^{m-1} &
  &\text{ for all } x\in{\mathcal B}_t,\ y\in{\mathcal F}_t,
\end{align*}
i.e., the quadrature approximation $\tilde g_t$ given by
(\ref{eq:green_g_apx}) is exponentially convergent with respect to $m$.
\end{theorem}
\begin{proof}
In order to apply Lemma~\ref{le:quadrature_product}, we have to construct
polynomials approximating $\hat f$ and $\hat g$.

Let $\hat f\in\{\hat f_1,\hat f_2\}$.
According to Lemma~\ref{le:derivatives}, we can apply
Theorem~\ref{th:chebyshev} to $\hat f$ using
\begin{align*}
  C_f &:= \frac{C_{\rm as}}{\delta_t^\sigma}, &
  \gamma_f &:= \frac{1}{2c_0}
\end{align*}
to obtain
\begin{equation*}
  \|\hat f - {\mathfrak I}_m^Q[\hat f]\|_{\infty,Q}
  \leq C_{\rm in}(m) C_f
        \varrho\left(\frac{1}{2 c_0}\right)^{-m}
\end{equation*}
with the polynomial $C_{\rm in}(m):=2 d e (m+2)^d (m+1) (1+4c_0)$.
We fix
\begin{align*}
  r &:= \frac{1}{2 c_0}, &
  \zeta &:= \frac{r+1}{\varrho(r)} < 1
\end{align*}
and find
\begin{align*}
  \|\hat f - {\mathfrak I}_m^Q[\hat f]\|_{\infty,Q}
  &\leq C_{\rm in}(m) C_f
         \varrho\left(\frac{1}{2 c_0}\right)^{-m}
   = \frac{C_{\rm in}(m) C_{\rm as}}{\delta_t^\sigma}
         \varrho(r)^{-m}\\
  &= \frac{C_{\rm in}(m) C_{\rm as}}{\delta_t^\sigma}
          \zeta^m (r+1)^{-m}.
\end{align*}
Due to $\zeta<1$, we can find a constant
\begin{equation*}
  C_{\rm apx} := \sup\{ C_{\rm in}(m) C_{\rm as} \zeta^m
                       \ :\ m\in\bbbn_0 \}
\end{equation*}
independent of $m$, $t$ and $\sigma$ such that
\begin{subequations}\label{eq:chebyshev_final}
\begin{align}
  \|\hat f - {\mathfrak I}_m^Q[\hat f]\|_{\infty,Q}
  &\leq \frac{C_{\rm apx}}{\delta_t^\sigma} (r+1)^{-m}
   = \frac{C_{\rm apx}}{\delta_t^\sigma}
        \left(\frac{2 c_0}{2 c_0 + 1}\right)^m &
  &\text{ for all } m\in\bbbn.
\end{align}
We can apply the same reasoning to $\hat g\in\{\hat g_1,\hat g_2\}$
to obtain
\begin{align}
  \|\hat g - {\mathfrak I}_{m-1}^Q[\hat g]\|_{\infty,Q}
  &\leq \frac{C_{\rm apx}'}{\delta_t^{\sigma+1}} (r+1)^{-m+1}
   = \frac{C_{\rm apx}'}{\delta_t^{\sigma+1}}
        \left(\frac{2 c_0}{2 c_0 + 1}\right)^{m-1} &
  &\text{ for all } m\in\bbbn
\end{align}
\end{subequations}
with a suitable constant $C_{\rm apx}'\in\bbbr_{\geq 0}$.

Now we can focus on the final estimate.
We have
\begin{subequations}
\begin{align}
  |g(x,y) &- \tilde g_t(x,y)|
   \leq \left| \int_{\partial\omega_t}
                    g(x,z)
                    \frac{\partial g}{\partial n(z)}(z,y) \,dz
             - \sum_{\nu=1}^k w_\nu
                    g(x,z_\nu)
                    \frac{\partial g}{\partial n(z_\nu)}(z_\nu,y) \right|
               \label{eq:quadrature_1}\\
  &\qquad + \left| \int_{\partial\omega_t}
                    \frac{\partial g}{\partial n(z)}(x,z)
                    g(z,y) \,dz
               - \sum_{\nu=1}^k w_\nu
                    \frac{\partial g}{\partial n(z_\nu)}(z_\nu,y)
                    g(x,z_\nu) \right|.
               \label{eq:quadrature_2}
\end{align}
\end{subequations}
For the first term, we use
\begin{align*}
  \int_{\partial\omega_t} g(x,z) \frac{\partial g}{\partial n(z)}(z,y) \,dz
  &= \sum_{\iota=1}^{2d} \sqrt{\det(D\gamma_\iota^* D\gamma_\iota)}
       \int_Q \hat f_1(\hat z) \hat g_1(\hat z) \,d\hat z,\\
  \sum_{\nu\in K} w_\nu g(x,z_\nu)
          \frac{\partial g}{\partial n(z_\nu)}(z_\nu,y)
  &= \sum_{\iota=1}^{2d} \sqrt{\det(D\gamma_\iota^* D\gamma_\iota)}
        \sum_{\mu\in M} \hat w_\mu \hat f_1(\hat x_\mu) \hat g_1(\hat x_\mu)
\end{align*}
to obtain
\begin{align*}
  &\int_{\partial\omega_t} g(x,z) \frac{\partial g}{\partial n(z)}(z,y) \,dz
     - \sum_{\nu\in K} w_\nu g(x,z_\nu)
               \frac{\partial g}{\partial n(z_\nu)}(z_\nu,y)\\
  &= \sum_{\iota=1}^{2d} \sqrt{\det(D\gamma_\iota^* D\gamma_\iota)}
       \left( \int_Q \hat f_1(\hat z) \hat g_1(\hat z) \,d\hat z
              - \sum_{\mu\in M} \hat w_\mu \hat f_1(\hat x_\mu)
                                  \hat g_1(\hat x_\mu) \right).
\end{align*}
Due to (\ref{eq:Phi_t}) and (\ref{eq:parametrization}), we have
\begin{equation*}
  \sqrt{\det(D\gamma_\iota^* D\gamma_\iota)}
  \leq (2\delta_t)^{d-1},
\end{equation*}
and we can use Lemma~\ref{le:quadrature_product} to bound the
second term and get
\begin{align*}
  &\left|\int_{\partial\omega_t} g(x,z)
               \frac{\partial g}{\partial n(z)}(z,y) \,dz
     - \sum_{\nu\in K} w_\nu g(x,z_\nu)
               \frac{\partial g}{\partial n(z_\nu)}(z_\nu,y)\right|\\
  &\leq (2\delta_t)^{d-1} 2^d
         ( \|\hat f_1-\hat p_1\|_{\infty,Q} \|\hat g_1\|_{\infty,Q}
         + \|\hat f_1\|_{\infty,Q} \|\hat g_1-\hat q_1\|_{\infty,Q}\\
  &\qquad + \|\hat f_1-\hat p_1\|_{\infty,Q} \|\hat g_1-\hat q_1\|_{\infty,Q} )
\end{align*}
for any $\hat p_1\in{\mathcal Q}_m$ and $\hat q_1\in{\mathcal Q}_{m-1}$.
It comes as no surprise that we use the tensor Chebyshev interpolation
polynomials $\hat p_1 := {\mathfrak I}_m^Q[\hat f_1]$ and
$\hat q_1 := {\mathfrak I}_{m-1}^Q[\hat g_1]$ investigated before.
The inequalities (\ref{eq:chebyshev_final}) provide us with estimates for
the interpolation error, while Lemma~\ref{le:omega_t} in combination
with (\ref{eq:asymptotically_smooth}) yields
\begin{align*}
  \|\hat f_1\|_{\infty,Q} &\leq \frac{C_{\rm as}}{\delta_t^\sigma}, &
  \|\hat g_1\|_{\infty,Q} &\leq \frac{C_{\rm as} c_0}{\delta_t^{\sigma+1}}.
\end{align*}
Combining both estimates we find
\begin{align*}
  &\left|\int_{\partial\omega_t} g(x,z)
              \frac{\partial g}{\partial n(z)}(z,y) \,dz
       - \sum_{\nu\in K} w_\nu g(x,z_\nu)
              \frac{\partial g}{\partial n(z_\nu)}(z_\nu,y)\right|\\
  &\leq 2^{2d-1} \delta_t^{d-1}
        \left( \frac{C_{\rm apx} C_{\rm as} c_0}{\delta_t^{2\sigma+1}}
                  \left(\frac{2c_0}{2c_0+1}\right)^m
               + \frac{C_{\rm apx}' C_{\rm as}}{\delta_t^{2\sigma+1}}
                  \left(\frac{2c_0}{2c_0+1}\right)^{m-1}\right.\\
  &\qquad \left. + \frac{C_{\rm apx} C_{\rm apx}'}{\delta_t^{2\sigma+1}}
                  \left(\frac{2c_0}{2c_0+1}\right)^{2m-1} \right)\\
  &\leq \frac{C_{\rm gr,1}}{\delta^{2\sigma-d+2}}
            \left(\frac{2c_0}{2c_0+1}\right)^{m-1}
\end{align*}
with the constant
\begin{equation*}
  C_{\rm gr,1} := 2^{2d-1}
   ( C_{\rm apx} C_{\rm as} c_0 + C_{\rm apx}' C_{\rm as}
       + C_{\rm apx} C_{\rm apx}' ).
\end{equation*}
We can obtain similar estimates for the second integrals
(\ref{eq:quadrature_2}) by exactly the same arguments and conclude
\begin{equation*}
  |g(x,y) - \tilde g_t(x,y)|
  \leq \frac{C_{\rm gr}}{\delta_t^{2\sigma-d+2}}
        \left(\frac{2c_0}{2c_0+1}\right)^{m-1}
\end{equation*}
with the constant $C_{\rm gr} := 4 d C_{\rm gr,1}$.
\qed
\end{proof}

\section{Cross approximation}

The construction outlined in the previous section still offers
room for improvement:
the matrices $A_{t,+}$ and $A_{t,-}$ appearing in (\ref{eq:ab_def})
describe the influence of Neumann and Dirichlet values of $g(\cdot,y)$
on $\partial\omega_t$ to the approximation in $\omega_t$.
Using the Poincar\'e-Steklov operator, we can construct the Neumann
values from the Dirichlet values, therefore we expect that it should
be possible to avoid using $A_{t,+}$ and thus reduce the rank of
our approximation by a factor of two.
Eliminating $A_{t,+}$ explicitly would require us to approximate
the Poincar\'e-Steklov operator and solve an integral equation on
the boundary $\partial\omega_t$, and we would have to reach a fairly
high accuracy in order to preserve the exponential convergence of
the quadrature approximation.

For the sake of efficiency, we choose an implicit approach:
assuming that $A_{t,+}$ can be obtained from $A_{t,-}$ by solving
a linear system, we expect that the rank of the matrix
$A_t=\begin{pmatrix} A_{t,+} & A_{t,-} \end{pmatrix}$ is lower than
the number of its columns.
Therefore we use an algebraic procedure to approximate
the matrix $A_t$ by a lower-rank matrix.
The \emph{adaptive cross approximation} approach \cite{BE00a,BERJ01,TY99}
is particularly attractive in this context, since it allows
us to construct an algebraic interpolation operator that can be
used to approximate matrix blocks based only on a few of their
entries.

The adaptive cross approximation of a matrix $X\in\bbbr^{\Idx\times\Jdx}$
is constructed as follows:
a pair of pivot elements $i_1\in\Idx$ and $j_1\in\Jdx$ are chosen
and the vectors $c^{(1)}\in\bbbr^\Idx$, $d^{(1)}\in\bbbr^\Jdx$ given by
\begin{align*}
  c^{(1)}_i &:= x_{i,j_1} / x_{i_1,j_1}, &
  d^{(1)}_j &:= x_{i_1,j} &
  &\text{ for all } i\in\Idx,\ j\in\Jdx
\end{align*}
are constructed.
The matrix
\begin{equation*}
  \widetilde X^{(1)} := c^{(1)} (d^{(1)})^*
\end{equation*}
satisfies
\begin{align*}
  \widetilde x^{(1)}_{i,j_1}
  &= x_{i,j_1} x_{i_1,j_1} / x_{i_1,j_1} = x_{i,j_1},\\
  \widetilde x^{(1)}_{i_1,j}
  &= x_{i_1,j_1} x_{i_1,j} / x_{i_1,j_1} = x_{i_1,j} &
  &\text{ for all } i\in\Idx,\ j\in\Jdx,
\end{align*}
i.e., it is identical to $X$ in the $i$-th row and the $j$-th column
(the eponymous ``cross'' formed by this row and column).
The remainder matrix
\begin{equation*}
  X^{(1)} := X - \widetilde X^{(1)}
\end{equation*}
therefore vanishes in this row and column.
If $X^{(1)}$ is considered small enough in a suitable sense, we use
$\widetilde X^{(1)}$ as a rank-one approximation of $X$.
Otherwise, we proceed by induction:
if $X^{(k)}$ is not sufficiently small for $k\in\bbbn$, we construct
a cross approximation $\widetilde X^{(k+1)} = c^{(k+1)} (d^{(k+1)})^*$ and
let
\begin{equation*}
  X^{(k+1)} := X^{(k)} - \widetilde X^{(k+1)}
  = X - \sum_{\nu=1}^{k+1} \widetilde X^{(\nu)}.
\end{equation*}
If $X^{(k)}$ is small enough, the matrix
\begin{equation*}
  \sum_{\nu=1}^k \widetilde X^{(\nu)}
   = \sum_{\nu=1}^k c^{(\nu)} (d^{(\nu)})^*
   = C D^*,
\end{equation*}
with
\begin{align*}
  C &:= \begin{pmatrix}
    c^{(1)} & \ldots & c^{(k)}
  \end{pmatrix}, &
  D &:= \begin{pmatrix}
    d^{(1)} & \ldots & d^{(k)}
  \end{pmatrix}
\end{align*}
is a rank-$k$ approximation of $X$. 

This approximation can be interpreted in terms of an algebraic
interpolation:
we introduce the matrix $P\in\bbbr^{k\times\Idx}$ by
\begin{align*}
  P z &:= \begin{pmatrix}
   z_{i_1}\\ \vdots\\ z_{i_k}
  \end{pmatrix} &
  &\text{ for all } z\in\bbbr^\Idx
\end{align*}
mapping a vector to the selected pivot elements.
The pivot elements play the role of interpolation points in our
reformulation of the cross approximation method.

We also need an equivalent of Lagrange polynomials.
Since the algorithm introduces zero rows and columns to the remainder
matrices
\begin{equation*}
  X^{(\ell)} = X - \widetilde X^{(1)} - \ldots - \widetilde X^{(\ell)}
\end{equation*}
and since the vectors $c^{(1)},\ldots,c^{(k)}$ and $d^{(1)},\ldots,d^{(k)}$
are just scaled columns and rows of the remainder matrices, we have
\begin{align*}
  c_{i\mu} &= c^{(\mu)}_i = 0 &
  &\text{ for all } i\in\{i_1,\ldots,i_{\mu-1}\},\\
  d_{j\mu} &= d^{(\mu)}_j = 0 &
  &\text{ for all } j\in\{j_1,\ldots,j_{\mu-1}\},
\end{align*}
and since the entries of $P C\in\bbbr^{k\times k}$ are given by
\begin{align*}
  (P C)_{\nu\mu} &= c^{(\mu)}_{i_\nu} &
  &\text{ for all } \nu,\mu\in\{1,\ldots,k\},
\end{align*}
this matrix is lower triangular.
Due to our choice of scaling, its diagonal elements are equal to
one, so the matrix is also invertible.

Therefore the matrix
\begin{equation*}
  V := C (P C)^{-1} \in\bbbr^{\Idx\times k}
\end{equation*}
is well-defined.
Its columns play the role of Lagrange polynomials, and the algebraic
interpolation operator given by
\begin{equation*}
  {\mathfrak I} := V P
\end{equation*}
satisfies the projection property
\begin{equation}\label{eq:aca_projection}
  {\mathfrak I} C = V P C = C (P C)^{-1} P C = C.
\end{equation}
Since the rows $i_1,\ldots,i_k$ in $X^{(k)}$ vanish, we have
\begin{equation*}
  0 = P X^{(k)} = P (X - C D^*)
\end{equation*}
and therefore
\begin{equation}\label{eq:aca_lowrank}
  {\mathfrak I} X
  = V P X
  = V P C D^*
  = C D^*,
\end{equation}
i.e., the low-rank approximation $C D^*$ results from algebraic
interpolation.

\section{Hybrid approximation}
\label{cha:hybrid_approximation}

The adaptive cross approximation algorithm gives us a powerful
heuristic method for constructing low-rank approximations of
arbitrary matrices.
In our case, we apply it to reduce the rank of the factorization given
by (\ref{eq:ab_def}), i.e., we apply the adaptive cross approximation
algorithm to the matrix $A_t\in\bbbr^{\hat t\times 2k}$ and obtain a reduced
rank $\ell\in\bbbn$ and matrices $C_t\in\bbbr^{\hat t\times\ell}$,
$D_t\in\bbbr^{2k \times \ell}$ and $P_t\in\bbbr^{\ell\times\hat t}$ such that
\begin{equation*}
  A_t \approx C_t D_t^* = V_t P_t A_t
\end{equation*}
with $V_t := C_t (P_t C_t)^{-1}$.
Combining this approximation with (\ref{eq:green_ab}) yields
\begin{equation*}
  G|_{\hat t\times\hat s}
  \approx A_t B_{ts}^*
  \approx C_t D_t^* B_{ts}^*
  = C_t (B_{ts} D_t)^*,
\end{equation*}
i.e., we have reduced the rank from $2k$ to $\ell$.

We can avoid computing the matrices $B_{ts}$ entirely by using algebraic
interpolation:
for ${\mathfrak I}_t := V_t P_t$, the projection property
(\ref{eq:aca_projection}) yields
\begin{equation*}
  G|_{\hat t\times\hat s}
  \approx C_t (B_{ts} D_t)^*
  = {\mathfrak I}_t C_t (B_{ts} D_t)^*
  \approx {\mathfrak I}_t G|_{\hat t\times\hat s}.
\end{equation*}
This approach has the advantage that we can prepare and store the
matrices $C_t\in\bbbr^{\hat t\times\ell}$, $P_t\in\bbbr^{\ell\times\hat t}$
and $P_t C_t\in\bbbr^{\ell\times\ell}$ in a setup phase.
Since $A_t$ depends only on the cluster $t$, but not on an entire
block, this phase involves only a sweep across the cluster tree
that does not lead to a large work-load.

Once the matrices have been prepared, an approximation of a block
$G|_{\hat t\times\hat s}$ can be found by computing its pivot rows
$P_t G|_{\hat t\times\hat s}$, obtaining $\widetilde B_{ts}$ through solving
the linear system
\begin{equation}\label{eq:hybrid_system}
  (P_t C_t) \widetilde B_{ts}^* = P_t G|_{\hat t\times\hat s}
\end{equation}
by forward substitution, and storing the rank-$\ell$-approximation
\begin{equation}\label{eq:hybrid}
  C_t \widetilde B_{ts}^* = C_t (P_t C_t)^{-1} P_t G|_{\hat t\times\hat s}
  = {\mathfrak I}_t G|_{\hat t\times\hat s}.
\end{equation}
None of these operations involves the quadrature rank $2k$, the
computational work is determined by the reduced rank $\ell$.

%
%
\begin{remark}[Complexity]
Computing the rank $\ell$ cross approximation of
$A_t\in\bbbr^{\hat t\times 2k}$ for one cluster $t\in\ctI$ requires
${\mathcal O}(\ell k \#\hat t)$ operations.
Similar to \cite[Lemma~2.4]{GRHA02}, we can conclude that the
cross approximation for all clusters requires not more than
${\mathcal O}(\ell k n \log n)\subseteq{\mathcal O}(\ell m^{d-1} n \log n)$
operations.

Solving the linear system (\ref{eq:hybrid_system}) for one block
$b=(t,s)\in\lfaIJ$ requires not more than ${\mathcal O}(\ell^2 \#\hat s)$
operations, and we can follow the reasoning of \cite[Lemma~2.9]{GRHA02}
to obtain a bound of ${\mathcal O}(\ell^2 n \log n)$ for the computational
work involved in setting up all blocks by the hybrid method.
\end{remark}

%
%
\begin{lemma}[Hybrid approximation]
Let $b=(t,s)\in\ctIJ$ be an admissible block.
Then we have
\begin{equation}\label{eq:hybrid_error_1}
  \| G|_{\hat t\times\hat s} - C_t \widetilde B_{ts}^* \|_2
  \leq (1 + \|{\mathfrak I}_t\|_2)
          \| G|_{\hat t\times\hat s} - A_t B_{ts}^* \|_2
       + \| A_t - C_t D_t^* \|_2 \| B_{ts}^* \|_2.
\end{equation}
\end{lemma}
\begin{proof}
Since (\ref{eq:aca_lowrank}) implies $C_t D_t^* = {\mathfrak I}_t A_t$,
we can use (\ref{eq:hybrid}) to obtain
\begin{align*}
  \| G|_{\hat t\times\hat s} &- C_t \widetilde B_{ts}^* \|_2
   = \| G|_{\hat t\times\hat s} - A_t B_{ts}^*
        + A_t B_{ts}^*
        - C_t D_t^* B_{ts}^*
        + C_t D_t^* B_{ts}^*
        - C_t \widetilde B_{ts}^* \|_2\\
  &\leq \| G|_{\hat t\times\hat s} - A_t B_{ts}^* \|_2
        + \| (A_t - C_t D_t^*) B_{ts}^* \|_2
        + \| {\mathfrak I}_t (A_t B_{ts}^* - G|_{\hat t\times\hat s}) \|_2\\
  &\leq \| G|_{\hat t\times\hat s} - A_t B_{ts}^* \|_2
        + \| A_t - C_t D_t^* \|_2 \| B_{ts}^* \|_2
        + \| {\mathfrak I}_t \|_2
            \| G|_{\hat t\times\hat s} - A_t B_{ts}^* \|_2\\
  &= (1 + \|{\mathfrak I}_t\|_2)
        \| G|_{\hat t\times\hat s} - A_t B_{ts}^* \|_2
        + \| A_t - C_t D_t^* \|_2 \| B_{ts}^* \|_2.
\end{align*}
\qed
\end{proof}

The error estimate (\ref{eq:hybrid_error_1}) contains two terms
that we can control directly:
the error of the analytical approximation
\begin{equation*}
  \|G|_{\hat t\times\hat s} - A_t B_{ts}^*\|_2
\end{equation*}
and the error of the cross approximation
\begin{equation*}
  \|A_t - C_t D_t^*\|_2.
\end{equation*}
The first error depends directly on the error of the kernel
approximation (cf. \cite[Section~4.6]{BO10} for a detailed analysis),
and Theorem~\ref{th:quadrature} allows us to reduce it to any chosen
accuracy, using results like \cite[Lemma~4.44]{BO10} to switch from
the maximum norm to the spectral norm.

The second error can be controlled directly by monitoring the
remainder matrices appearing in the cross approximation algorithm.

Therefore we only have to address the additional factors
$1+\|{\mathfrak I}_t\|_2$ and $\|B_{ts}^*\|_2$.

The norm $\|{\mathfrak I}_t\|_2$ is the algebraic counterpart of
the Lebesgue constant of standard interpolation methods.
We can monitor this norm explicitly:
since $P_t$ is surjective, we have $\|{\mathfrak I}_t\|_2=\|V_t\|_2$
and can obtain bounds for this matrix during the construction of
the cross approximation.
Should the product
\begin{equation*}
  (1+\|{\mathfrak I}_t\|_2)
  \|G|_{\hat t\times\hat s} - A_t B_{ts}^*\|_2
\end{equation*}
become too large, we can increase the accuracy of the analytic
approximation of the kernel function to reduce the second term.
A common (as far as we know still unproven) assumption in the
field of cross approximation methods states that
$\|{\mathfrak I}_t\|_2\sim\ell^\alpha$ holds for a small $\alpha>0$
if a suitable pivoting strategy is employed, cf. \cite[eq. (16)]{BEVE12}.

For the analysis of the factor $\|B_{ts}^*\|_2$, we have to take
the choice of basis functions into account.
For the sake of simplicity, we assume that the finite element basis
is stable in the sense that there is a constant $C_\psi\in\bbbr_{>0}$,
possibly depending on the underlying grid, such that
\begin{align}\label{eq:fe_stability}
  \left\| \sum_{j\in\Jdx} u_j \psi_j \right\|_{L^2}
  &\leq C_\psi \|u\|_2 &
  &\text{ for all } u\in\bbbr^\Jdx.
\end{align}

%
%
\begin{lemma}[Scaling factor]
\label{le:hybrid_scaling}
There is a constant $C_{\rm sf}\in\bbbr_{>0}$ depending only on
$C_{\rm as}$, $c_0$, $|\Omega|$ and $d$ such that
\begin{align*}
  \| B_{ts}^* \|_2 &\leq \frac{C_{\rm sf} C_\psi}{\delta_t^{\sigma-d/2+3/2}} &
  &\text{ for all } (t,s)\in\ctI\times\ctJ
       \text{ with } {\mathcal B}_s\subseteq{\mathcal F}_t.
\end{align*}
\end{lemma}
\begin{proof}
By definition (\ref{eq:green_ab}), we have
\begin{equation*}
  B_{ts} = \begin{pmatrix} B_{ts+} & B_{ts-} \end{pmatrix}
\end{equation*}
and therefore
\begin{equation*}
  \| B_{ts}^* \|_2
  = \left\| \begin{pmatrix} B_{ts+}^*\\ B_{ts-}^* \end{pmatrix} \right\|_2
  \leq \sqrt{\|B_{ts+}^*\|_2^2 + \|B_{ts-}^*\|_2^2}.
\end{equation*}
We focus on $B_{ts+}$.
A simple application of the Cauchy-Schwarz inequality yields
\begin{equation*}
  \| B_{ts+}^* \|_2
  = \sup_{\substack{u\in\bbbr^{\hat s}\setminus\{0\}\\
                    v\in\bbbr^K\setminus\{0\}}}
     \frac{\langle B_{ts+}^* u, v \rangle_2}{\|u\|_2 \|v\|_2}.
\end{equation*}
Let $u\in\bbbr^{\hat s}$ and $v\in\bbbr^K$.
We use the definition (\ref{eq:ab_def}) and apply the Cauchy-Schwarz
inequality first to the inner product in $L^2(\Omega)$ and then to the
Euclidean one to get
\begin{align*}
  \langle B_{ts+}^* u, v \rangle_2
  &= \sum_{j\in\hat s} \sum_{\nu\in K} b_{ts,j\nu} u_j v_\nu
   = \int_\Omega \sum_{j\in\hat s} u_j \psi_j(y)
           \sum_{\nu\in K} v_\nu \sqrt{w_\nu}
           \frac{\partial g}{\partial n(z_\nu)}(z_\nu,y) \,dy\\
  &\leq \left( \int_\Omega \left(\sum_{j\in\hat s} u_j \psi_j(y)\right)^2 \,dy
          \right)^{1/2}\\
  &\qquad \left( \int_\Omega \left( \sum_{\nu\in K} v_\nu \sqrt{w_\nu}
             \frac{\partial g}{\partial n(z_\nu)}(z_\nu,y)\right)^2 \,dy
          \right)^{1/2}\\
  &= \left\| \sum_{j\in\hat s} u_j \psi_j \right\|_{L^2}
     \left( \int_\Omega \left( \sum_{\nu\in K} v_\nu \sqrt{w_\nu}
             \frac{\partial g}{\partial n(z_\nu)}(z_\nu,y)\right)^2 \,dy
          \right)^{1/2}\\
  &\leq C_\psi \|u\|_2
     \left( \int_\Omega \left( \sum_{\nu\in K} v_\nu^2 \right)
               \left( \sum_{\nu\in K} w_\nu
               \left(\frac{\partial g}{\partial n(z_\nu)}(z_\nu,y)\right)^2
               \right) \,dy \right)^{1/2}\\
  &= C_\psi \|u\|_2 \|v\|_2
         \left( \int_\Omega \sum_{\nu\in K} w_\nu
              \left( \frac{\partial g}{\partial n(z_\nu)}(z_\nu,y) \right)^2
              \,dy \right)^{1/2}.
\end{align*}
The asymptotic smoothness (\ref{eq:asymptotically_smooth}) in combination
with Lemma~\ref{le:omega_t} yields
\begin{equation*}
  \left| \frac{\partial g}{\partial n(z_\nu)}(z_\nu,y) \right|
  \leq \frac{C_{\rm as} c_0}{\|z_\nu-y\|^{\sigma+1}}
  \leq \frac{C_{\rm as} c_0}{\delta_t^{\sigma+1}}.
\end{equation*}
Since the weights are non-negative and the quadrature rule integrates
constants exactly, we have
\begin{equation*}
  \sum_{\nu\in K} w_\nu \left( \frac{\partial g}{\partial n(z_\nu)}(z_\nu,y)
                    \right)^2
  \leq \sum_{\nu\in K} w_\nu \left(\frac{C_{\rm as} c_0}
                                    {\delta_t^{\sigma+1}}\right)^2
  = |\partial\omega_t| \left(\frac{C_{\rm as} c_0}
                                  {\delta_t^{\sigma+1}}\right)^2.
\end{equation*}
We conclude
\begin{align*}
  \langle B_{ts+}^* u, v \rangle_2
  &\leq C_\psi \|u\|_2 \|v\|_2 \left( \int_\Omega |\partial\omega_t|
                  \left(\frac{C_{\rm as} c_0}{\delta_t^{\sigma+1}}\right)^2
                  \,dy \right)^{1/2}\\
  &= \frac{C_\psi C_{\rm as} c_0}{\delta_t^{\sigma+1}}
            \sqrt{|\Omega| |\partial\omega_t|} \|u\|_2 \|v\|_2.
\end{align*}
Inserting $|\partial\omega_t|\leq 2d (4\delta_t)^{d-1} = 2^{2d-1} d
\delta_t^{d-1}$ leads to
\begin{equation*}
  \langle B_{ts+}^* u, v \rangle_2
  \leq \frac{C_\psi C_{\rm as} c_0}{\delta_t^{\sigma+1}}
        \sqrt{|\Omega| 2^{2d-1} d} \delta_t^{d/2-1/2}
        \|u\|_2 \|v\|_2.
\end{equation*}
For $B_{ts-}$, we obtain the same result, since the definition
(\ref{eq:ab_def}) includes the scaling factor $1/\delta_t$.
Combining both estimates and choosing the constant
$C_{\rm sf} := C_{\rm as} c_0 \sqrt{|\Omega| 2^{2d} d}$ completes
the proof.
\qed
\end{proof}

\section{Nested cross approximation}

Applying the algorithm presented so far to admissible blocks
yields a hierarchical matrix approximation of $G$.
We would prefer to obtain an ${\mathcal H}^2$-matrix due to its
significantly lower complexity.

Since the cluster basis has to be able to handle \emph{all}
blocks connected to a given cluster, we have to approximate the
entire farfield.
We let
\begin{align*}
  F_t &:= \{ j\in\Jdx\ :\ \supp\psi_j\subseteq {\mathcal F}_t \} &
  &\text{ for all } t\in\ctI
\end{align*}
and use our algorithm to find low-rank interpolation operators
${\mathfrak I}_t = V_t P_t$ such that
\begin{equation*}
  G|_{\hat t\times F_t}
  \approx {\mathfrak I}_t G|_{\hat t\times F_t}
  = V_t P_t G|_{\hat t\times F_t}.
\end{equation*}
Since our construction of ${\mathfrak I}_t$ depends only on the
matrix $A_t$, this approach does not increase the computational
work.

In order to obtain a \emph{uniform} hierarchical matrix, we
also require an approximation of the column clusters.
Our algorithm can easily handle this task as well:
as before, we let
\begin{align*}
  F_s &:= \{ i\in\Idx\ :\ \supp\varphi_i\subseteq {\mathcal F}_s \} &
  &\text{ for all } s\in\ctJ
\end{align*}
and use our algorithm with the \emph{adjoint} matrix to find
low-rank interpolation operators ${\mathfrak I}_s = V_s P_s$ such
that
\begin{equation*}
  G|_{F_t\times\hat s}^*
  \approx {\mathfrak I}_s G|_{F_t\times\hat s}^*
  = V_s P_s G|_{F_t\times\hat s}^*.
\end{equation*}
If a block $b=(t,s)$ satisfies the admissibility condition
\begin{align}\label{eq:admiss_grh_h2}
  \nonumber (t,s) \text{ admissible}
  &\iff ({\mathcal B}_s \subseteq {\mathcal F}_t \wedge
         {\mathcal B}_t \subseteq {\mathcal F}_s)\\
  &\iff \max\{\diam_\infty({\mathcal B}_t),
              \diam_\infty({\mathcal B}_s) \}
        \leq \dist_\infty({\mathcal B}_t,{\mathcal B}_s),
\end{align}
we have $\hat s\subseteq F_t$ and $\hat t\subseteq F_s$ and
therefore obtain
\begin{equation}
  \label{eqn:green_h2}
  G|_{\hat t\times\hat s}
  \approx {\mathfrak I}_t G|_{\hat t\times\hat s}
  \approx {\mathfrak I}_t G|_{\hat t\times\hat s} {\mathfrak I}_s^*
  = V_t P_t G|_{\hat t\times\hat s} P_s^* V_s^*
  = V_t S_b V_s^*
\end{equation}
with $S_b := P_t G|_{\hat t\times\hat s} P_s^*$.
We have found a way to construct a uniform hierarchical matrix.
Note that we can compute $S_b$ by evaluating $G$ in the small
number of pivot elements chosen for the row and column clusters.

In order to obtain an ${\mathcal H}^2$-matrix, the cluster bases have
to be nested.
Similar to the procedure outlined in \cite{BEVE12}, we only have
to ensure that the pivot elements for clusters with sons are
chosen among the sons' pivot elements.

For the sake of simplicity, we consider only the case of a
binary cluster tree and construct the cluster basis recursively.
Let $t\in\ctI$.
If $t$ is a leaf, i.e., if $\sons(t)=\emptyset$, we use our algorithm
as before and obtain ${\mathfrak I}_t = V_t P_t$.

If $t$ is not a leaf, we have $\sons(t)=\{t_1,t_2\}$.
We assume that we have already found ${\mathfrak I}_{t_1} = V_{t_1} P_{t_1}$
and ${\mathfrak I}_{t_2} = V_{t_2} P_{t_2}$ with
\begin{align*}
  G|_{t_1\times F_{t_1}}
  &\approx {\mathfrak I}_{t_1} G|_{t_1\times F_{t_1}}, &
  G|_{t_2\times F_{t_2}}
  &\approx {\mathfrak I}_{t_2} G|_{t_1\times F_{t_2}}
\end{align*}
by recursion.
Due to $F_t \subseteq F_{t_1} \cap F_{t_2}$, we have
\begin{align*}
  G|_{\hat t\times F_t}
  &= \begin{pmatrix}
       G|_{\hat t_1\times F_t}\\
       G|_{\hat t_2\times F_t}
     \end{pmatrix}
   \approx \begin{pmatrix}
       {\mathfrak I}_{t_1} G|_{\hat t_1\times F_t}\\
       {\mathfrak I}_{t_2} G|_{\hat t_2\times F_t}
     \end{pmatrix}
   = \begin{pmatrix}
       V_{t_1} P_{t_1} G|_{\hat t_1\times F_t}\\
       V_{t_2} P_{t_2} G|_{\hat t_2\times F_t}
     \end{pmatrix}\\
  &= \begin{pmatrix}
       V_{t_1} & \\
       & V_{t_2}
     \end{pmatrix}
     \begin{pmatrix}
       P_{t_1} G|_{\hat t_1\times F_t}\\
       P_{t_2} G|_{\hat t_2\times F_t}
     \end{pmatrix}
   = \begin{pmatrix}
       V_{t_1} & \\
       & V_{t_2}
     \end{pmatrix}
     \begin{pmatrix}
       P_{t_1} & \\
       & P_{t_2}
     \end{pmatrix} G|_{\hat t\times F_t}\\
  &\approx \begin{pmatrix}
       V_{t_1} & \\
       & V_{t_2}
     \end{pmatrix}
     \begin{pmatrix}
       P_{t_1} & \\
       & P_{t_2}
     \end{pmatrix} A_t B_{tF_t}^{*},
\end{align*}
where $A_t$ and $B_{tF_t}$ are again the matrices of the quadrature
method.
By applying the cross approximation algorithm to the two middle
factors, we find $\widehat{\mathfrak I}_t = \widehat V_t \widehat P_t$
such that
\begin{equation*}
  \begin{pmatrix}
    P_{t_1} & \\
    & P_{t_2}
  \end{pmatrix} A_t
  \approx \widehat V_t \widehat P_t
  \begin{pmatrix}
    P_{t_1} & \\
    & P_{t_2}
  \end{pmatrix} A_t.
\end{equation*}
All we have to do is to let
\begin{align*}
  V_t &:= \begin{pmatrix}
    V_{t_1} & \\
    & V_{t_2}
  \end{pmatrix} \widehat V_t, &
  P_t &:= \widehat P_t \begin{pmatrix}
    P_{t_1} & \\
    & P_{t_2}
  \end{pmatrix}
\end{align*}
and observe
\begin{align*}
  G|_{\hat t\times F_t}
  &\approx \begin{pmatrix}
       V_{t_1} & \\
       & V_{t_2}
     \end{pmatrix}
     \begin{pmatrix}
       P_{t_1} & \\
       & P_{t_2}
     \end{pmatrix} A_t B_{tF_t}^{*}\\
  &\approx \begin{pmatrix}
       V_{t_1} & \\
       & V_{t_2}
     \end{pmatrix}
     \widehat V_t \widehat P_t
     \begin{pmatrix}
       P_{t_1} & \\
       & P_{t_2}
     \end{pmatrix} A_t B_{tF_t}^{*}
   = V_t P_t A_t B_{tF_t}^{*}\\
  &\approx V_t P_t G|_{\hat t\times F_t}
   = {\mathfrak I}_t G|_{\hat t\times F_t}.
\end{align*}
Splitting $\widehat V_t$ into an upper and a lower part matching
$V_{t_1}$ and $V_{t_2}$ yields
\begin{align*}
  \begin{pmatrix}
    E_{t_1}\\
    E_{t_2}
  \end{pmatrix}
  &:= \widehat V_t, &
  V_t
  &= \begin{pmatrix}
       V_{t_1} & \\
       & V_{t_2}
     \end{pmatrix} \widehat V_t
   = \begin{pmatrix}
       V_{t_1} & \\
       & V_{t_2}
     \end{pmatrix}
     \begin{pmatrix}
       E_{t_1}\\ E_{t_2}
     \end{pmatrix}
   = \begin{pmatrix}
       V_{t_1} E_{t_1}\\
       V_{t_2} E_{t_2}
     \end{pmatrix},
\end{align*}
so we have indeed found a \emph{nested} cluster basis.

%
%
\begin{remark}[Complexity]
We assume that the ranks obtained by the cross approximation are
bounded by $\ell$.

For a leaf cluster, the construction of $V_t$ and $P_t$ requires
${\mathcal O}(\ell k \#\hat t)$ operations.
For a non-leaf cluster, the cross approximation is applied
to a $\min\{\#\hat t,2\ell\}\times(2k)$-matrix and takes no more than
${\mathcal O}(\ell k \min\{\#\hat t,2\ell\})$ operations.
Similar to \cite[Remark~4.1]{BOHA02}, we conclude that
${\mathcal O}(\ell k n)$ operations are sufficient to set up the
cluster bases.

The coupling matrices require us to solve two linear systems by
forward substitution, one with the matrix $P_t C_t$ and one with
the matrix $P_s C_s$.
The first is of dimension $\min\{\ell,\#\hat t\}$, the second of
dimension $\min\{\ell,\#\hat s\}$, therefore solving both systems
takes not more than
${\mathcal O}(\min\{\ell,\#\hat t\}^2 \ell
+\min\{\ell,\#\hat s\}^2 \ell)$ operations for
one block $b=(t,s)\in\lfaIJ$.
As in \cite[Remark~4.1]{BOHA02}, we conclude that not more than
${\mathcal O}(\ell^2 n)$ operations are required to compute all
coupling matrices and that the coupling matrices require not
more than ${\mathcal O}(\ell n)$ units of storage.
\end{remark}

Since the recursive algorithm applies multiple approximations to
the same block, we have to take a closer look at error estimates.
The total error in a given cluster is influenced by the errors
introduced in its sons, their sons, and so on.
In order to handle these connections, we introduce the set of
\emph{descendents} of a cluster $t\in\ctI$ by
\begin{align*}
  \sons^*(t) &:= \begin{cases}
    \{t\} &\text{ if } \sons(t)=\emptyset,\\
    \{t\} \cup \sons^*(t_1) \cup \sons^*(t_2)
    &\text{ if } \sons(t)=\{t_1,t_2\}.
  \end{cases}
\end{align*}
Each step of the algorithm introduces an error for the current
cluster:
for leaves, we approximate $G|_{\hat t\times F_t}$, while for non-leaves
the restriction of $G|_{\hat t\times F_t}$ to the sons' pivot elements
is approximated.
We denote the error added by quadrature and cross approximation
in each cluster $t\in\ctI$ by
\begin{align*}
  \hat\epsilon_t &:= \begin{cases}
    \|G|_{\hat t\times F_t} - {\mathfrak I}_t G|_{\hat t\times F_t}\|_2
    &\text{ if } \sons(t)=\emptyset,\\
    \left\| \begin{pmatrix}
      P_{t_1} G|_{\hat t_1\times F_t}\\
      P_{t_2} G|_{\hat t_2\times F_t}
    \end{pmatrix}
    - \widehat{\mathfrak I}_t
    \begin{pmatrix}
      P_{t_1} G|_{\hat t_1\times F_t}\\
      P_{t_2} G|_{\hat t_2\times F_t}
    \end{pmatrix} \right\|_2
    &\text{ if } \sons(t)=\{t_1,t_2\}.
  \end{cases}
\end{align*}
We have already seen that we can control these ``local'' errors
by choosing the quadrature order and the error tolerance of the
cross approximation appropriately.
Combining these estimates with stability estimates yields the
following bound for the global error.

%
%
\begin{lemma}[Error estimate]
Using the stability constants
\begin{align*}
  \widehat\Lambda_t
  &:= \begin{cases}
    1 &\text{ if } \sons(t)=\emptyset,\\
    \max\{\|V_{t_1}\|_2, \|V_{t_2}\|_2 \}
    &\text{ if } \sons(t)=\{t_1,t_2\}
  \end{cases} &
  &\text{ for all } t\in\ctI,
\end{align*}
the total approximation error can be bounded by
\begin{align}\label{eq:recursion_1}
  \| G|_{\hat t\times F_t} - {\mathfrak I}_t G|_{\hat t\times F_t} \|_2
  &\leq \sum_{r\in\sons^*(t)} \widehat\Lambda_t \hat\epsilon_t &
  &\text{ for all } t\in\ctI.
\end{align}
\end{lemma}
\begin{proof}
By structural induction.

Let $t\in\ctI$ be a leaf of the cluster tree.
Then we have $\sons^*(t)=\{t\}$ and (\ref{eq:recursion_1}) holds by
definition.

Let now $t\in\ctI$ be a cluster with $\sons(t)=\{t_1,t_2\}$ and assume
that (\ref{eq:recursion_1}) holds for $t_1$ and $t_2$.
We have
\begin{align*}
  \| G|_{\hat t\times F_t} - {\mathfrak I}_t G|_{\hat t\times F_t} \|_2
  &= \left\| \begin{pmatrix}
       G|_{\hat t_1\times F_t}\\
       G|_{\hat t_2\times F_t}
     \end{pmatrix}
     - \begin{pmatrix}
       V_{t_1} & \\
       & V_{t_2}
     \end{pmatrix}
     \widehat{\mathfrak I}_t
     \begin{pmatrix}
       P_{t_1} & \\
       & P_{t_2}
     \end{pmatrix}
     \begin{pmatrix}
       G|_{\hat t_1\times F_t}\\
       G|_{\hat t_2\times F_t}
     \end{pmatrix} \right\|_2\\
  &= \left\| \begin{pmatrix}
       G|_{\hat t_1\times F_t}\\
       G|_{\hat t_2\times F_t}
     \end{pmatrix}
     - \begin{pmatrix}
       V_{t_1} & \\
       & V_{t_2}
     \end{pmatrix}
     \begin{pmatrix}
       P_{t_1} G|_{\hat t_1\times F_t}\\
       P_{t_2} G|_{\hat t_2\times F_t}
     \end{pmatrix}\right.\\
  &\left. + \begin{pmatrix}
       V_{t_1} & \\
       & V_{t_2}
     \end{pmatrix}
     \begin{pmatrix}
       P_{t_1} G|_{\hat t_1\times F_t}\\
       P_{t_2} G|_{\hat t_2\times F_t}
     \end{pmatrix}
     - \begin{pmatrix}
       V_{t_1} & \\
       & V_{t_2}
     \end{pmatrix}
     \widehat{\mathfrak I}_t
     \begin{pmatrix}
       P_{t_1} G|_{\hat t_1\times F_t}\\
       P_{t_2} G|_{\hat t_2\times F_t}
     \end{pmatrix} \right\|_2\\
  &\leq \left\| \begin{pmatrix}
       G|_{\hat t_1\times F_t} - {\mathfrak I}_{t_1} G|_{\hat t_1\times F_t}\\
       G|_{\hat t_2\times F_t} - {\mathfrak I}_{t_2} G|_{\hat t_2\times F_t}
     \end{pmatrix} \right\|_2\\
  &+ \left\| \begin{pmatrix}
       V_{t_1} & \\
       & V_{t_2}
     \end{pmatrix} \right\|_2
     \left\| \begin{pmatrix}
       P_{t_1} G|_{\hat t_1\times F_t}\\
       P_{t_2} G|_{\hat t_2\times F_t}
     \end{pmatrix}
     - \widehat{\mathfrak I}_t
     \begin{pmatrix}
       P_{t_1} G|_{\hat t_1\times F_t}\\
       P_{t_2} G|_{\hat t_2\times F_t}
     \end{pmatrix} \right\|_2\\
  &\leq \| G|_{\hat t_1\times F_t}
           - {\mathfrak I}_{t_1} G|_{\hat t_1\times F_{t_1}} \|_2\\
  &\qquad + \| G|_{\hat t_2\times F_t}
           - {\mathfrak I}_{t_2} G|_{\hat t_2\times F_{t_2}} \|_2
     + \widehat\Lambda_t \hat\epsilon_t.
\end{align*}
The induction assumption yields
\begin{equation*}
  \|G|_{\hat t\times F_t}
    - {\mathfrak I}_t G|_{\hat t\times F_t}\|_2
  \leq \sum_{r\in\sons^*(t_1)} \widehat\Lambda_r \hat\epsilon_r
       + \sum_{r\in\sons^*(t_2)} \widehat\Lambda_r \hat\epsilon_r
       + \widehat\Lambda_t \hat\epsilon_t
  = \sum_{r\in\sons^*(t)} \widehat\Lambda_t \hat\epsilon_t,
\end{equation*}
and the induction is complete.
\qed
\end{proof}

We can find upper bounds for the stability constants during
the course of the algorithm:
let $t\in\ctI$.
If $\sons(t)=\emptyset$, the matrix $V_t$ can be constructed
explicitly by our algorithm, so we can either find $\|V_t\|_2$
by computing a singular value decomposition or obtain a good
estimate by using a power iteration.

If $\sons(t)=\{t_1,t_2\}$, we can use
\begin{equation}\label{eq:recursion_Vt}
  \|V_t\|_2
  = \left\| \begin{pmatrix}
      V_{t_1} & \\
      & V_{t_2}
    \end{pmatrix} \widehat V_t \right\|_2
  \leq \left\| \begin{pmatrix}
      V_{t_1} & \\
      & V_{t_2}
    \end{pmatrix} \right\|_2
    \| \widehat V_t \|_2
  = \max\{ \|V_{t_1}\|_2, \|V_{t_2}\|_2 \} \|\widehat V_t\|_2
\end{equation}
to compute an estimate of $\|V_t\|_2$ based on an estimate of
the small matrix $\widehat V_t$ that can be treated as before.

This approach allows us to obtain estimates for $\widehat\Lambda_t$
that can be computed explicitly during the course of the algorithm
and used to verify that (\ref{eq:recursion_1}) is bounded.

An alternative approach can be based on the conjecture
\cite[eq. (16)]{BEVE12}:
let $\ell\in\bbbn$ denote an upper bound for the rank used in the
cross approximation algorithms.
If we assume that there is a constant $\Lambda_V\geq 1$ such that
\begin{align*}
  \| V_t \| &\leq \Lambda_V \ell &
  &\text{ for all } t\in\lfI,\\
  \| \widehat V_t \|_2 &\leq \Lambda_V \ell &
  &\text{ for all } t\in\ctI\setminus\lfI,
\end{align*}
a simple induction using the inequality (\ref{eq:recursion_Vt})
immediately yields
\begin{align*}
  \| V_t \|_2 &\leq \Lambda_V^p \ell^p &
  &\text{ for all } t\in\ctI,
\end{align*}
where $p\in\bbbn$ denotes the depth of the cluster tree $\ctI$.
Since $\|V_t\|_2$ grows only polynomially with $\ell$ while the
error $\hat\epsilon_t$ converges exponentially, the right-hand
side of the error estimate (\ref{eq:recursion_1}) will also
converge exponentially.

\section{Numerical experiments}

The theoretical properties of the new approximation method, which we
will call \emph{Green hybrid method (GrH)} in the following, have been
discussed in detail in the preceding sections.
We will now investigate how the new method performs in experiments.

We consider the direct boundary element formulation of the Dirichlet
problem:
let $f$ be a harmonic function in $\Omega$ and assume that its
Dirichlet values $f|_{\partial\Omega}$ are given.
Solving the integral equation
\begin{align*}
  \int_{\partial\Omega} g(x,y) \frac{\partial f}{\partial n}(y) \,dy
  &= \frac{1}{2} f(x)
     + \int_{\partial\Omega} \frac{\partial g}{\partial n(y)}(x,y) f(y) \,dy &
  &\text{ for almost all } x\in\partial\Omega
\end{align*}
yields the Neumann values $\frac{\partial f}{\partial n}|_{\partial\omega}$.
We set up the Galerkin matrices $V$ and $K$ for the single and double
layer potential operators as well as the mass-matrix $M$ and solve the
equation
\begin{equation*}
 V \alpha = \left(K + \frac{1}{2} M \right) \beta,
\end{equation*}
where $\beta$ are the coefficients of the $L^2$-projection for the given
Dirichlet data in the piecewise linear basis $(\psi_j)_{j\in\Jdx}$ and
$\alpha$ are the coefficients for the desired Neumann data in the piecewise
constant basis $(\varphi_i)_{i\in\Idx}$.

For testing purpose we use the following three harmonic functions:
\begin{equation*}
 f_1(x) = x_1^2 - x_3^2, \quad f_2(x) = g(x, (1.2, 1.2, 1.2)), \quad
 f_3(x) = g(x, (1.0, 0.25, 1.0)) .
\end{equation*}
The approximation quality is measured by the absolute $L^2$-error of
the Neumann data
\begin{equation*}
 \epsilon_j
  = \left( \int_{\partial \Omega} \left( \frac{\partial}{\partial n} f_j(x)
   - \sum_{i \in \Idx} \alpha_i \varphi_i(x) \right)^2 \mathrm d x \right)^{1/2}.
\end{equation*}
The parameters for the Green hybrid method and for the adaptive cross
approximation have been chosen manually to ensure that the total error
(resulting from quadrature, matrix compression, and discretization) is
close to the discretization error for all three harmonic functions.

Nearfield entries are computed using Sauter's quadrature rule
\cite{SA96,SASC04} with 3 Gauss points per dimension for regular
integrals and 5 Gauss points for singular integrals.

All computations are performed on a single AMD Opteron 8431 Core at 2.4GHz.
In all numerical experiments we compare our new approach with the well-known
adaptive cross approximation technique (ACA) \cite[Algorithm~4.2]{BERJ01}.

%
%
\begin{figure}[Ht!]
 \begin{center}
  \subfigure[Memory per degree of freedom]{\label{fig:greenhybrid_sub_mem}
            \includegraphics[width=0.47\textwidth]{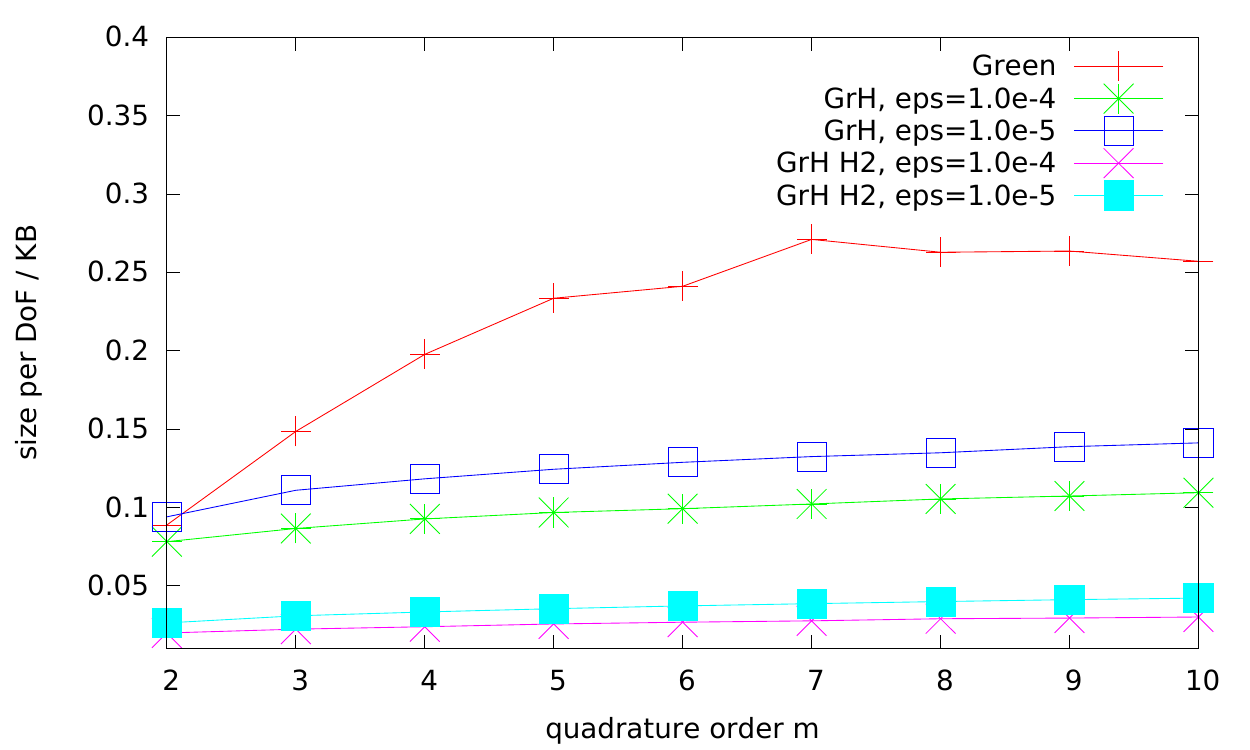}
        }
  \subfigure[Relative error]{\label{fig:greenhybrid_sub_error}
            \includegraphics[width=0.47\textwidth]{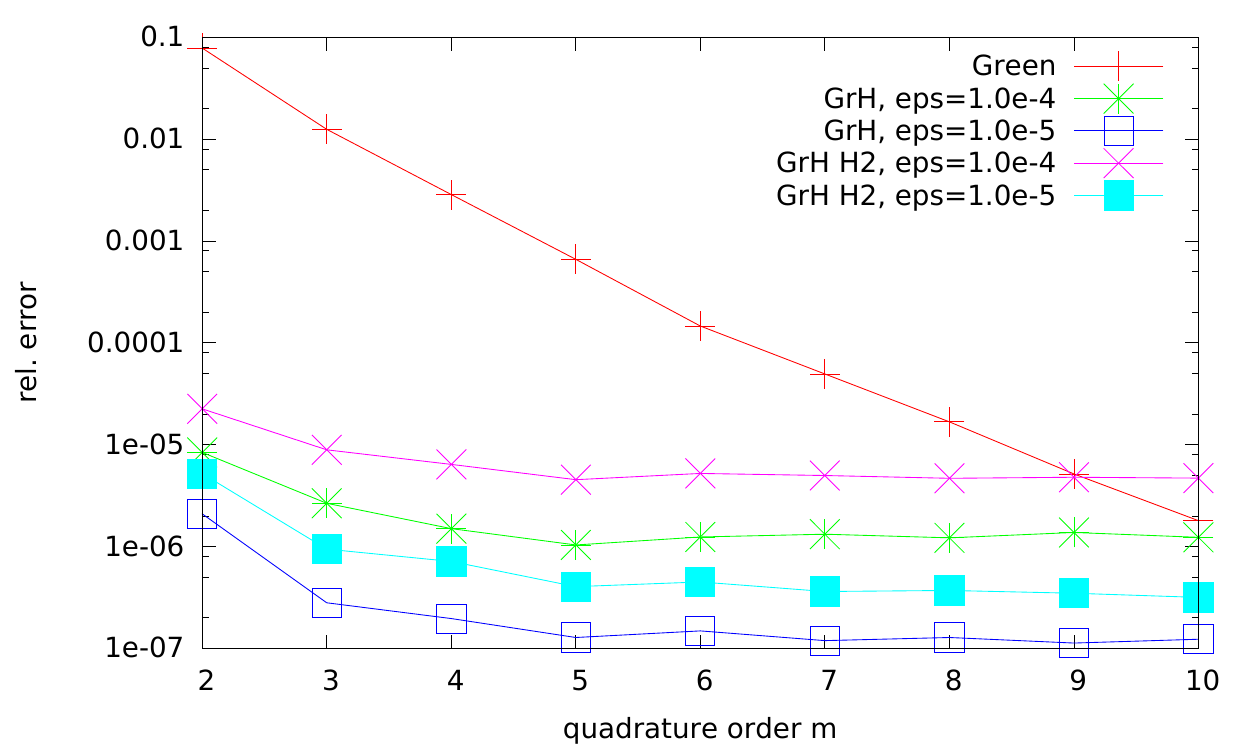}
        }
 \end{center}
  \subfigure[time per degree of freedom]{\label{fig:greenhybrid_sub_time}
            \includegraphics[width=0.47\textwidth]{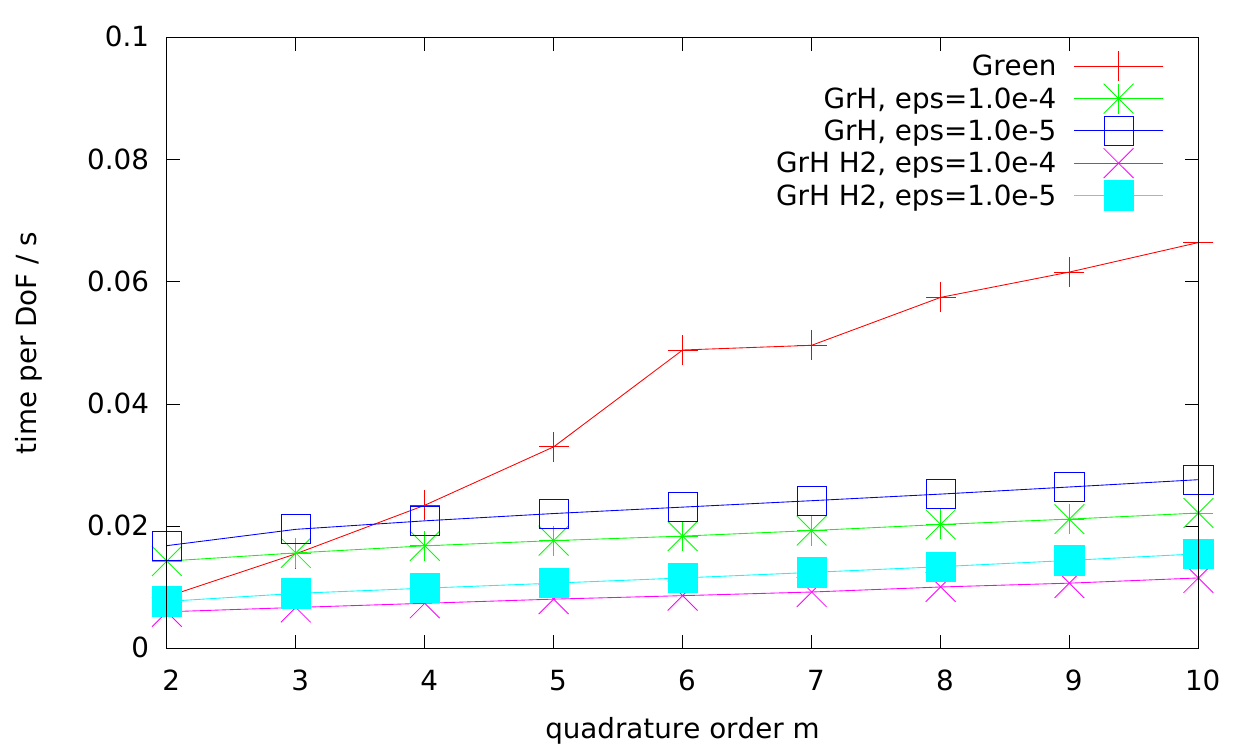}
        }
 \caption{Unit sphere, 32768 triangles:
       memory consumption, relative error and setup time for SLP with
       Green method vs. Green hybrid method vs. Green hybrid
       method $\mathcal H^2$.
       In all cases $\delta_t = \diam_\infty({\mathcal B}_t)/2$ was chosen.}
 \label{img:cmp_green_greenhybrid}
\end{figure}

\paragraph*{Reduced ranks.}\quad
The pure quadrature approximation \eqref{eq:green_ab} produces matrices
with local rank of $2k = 12m^2$. This implies that for $m=10$ the local rank
already reaches a value of 1200, leading to a rather unattractive
compression rate.
As we stated in the beginning of chapter \ref{cha:hybrid_approximation},
a rank reduction from $2k$ down to $k$ is at least expected when applying
cross approximation to the quadrature approach from 
\eqref{eq:green_ab} due to the linear dependency of Neumann and Dirichlet
values.
Figure~\ref{img:cmp_green_greenhybrid} shows that the Green hybrid method
\eqref{eq:hybrid} performs even better:
the storage requirements, given in Figure~\ref{fig:greenhybrid_sub_mem},
are reduced by approximately 50\%, and the $\mathcal{H}^2$-matrix version
\eqref{eqn:green_h2} reduces the storage requirements by approximately 75\%
compared to the $\mathcal{H}$-matrix version..

Figure~\ref{fig:greenhybrid_sub_error} illustrates that the pure Green
quadrature approach \eqref{eq:green_ab} leads to exponential convergence,
as predicted by Theorem~\ref{th:quadrature}.
The hybrid methods reach a surprisingly high accuracy even for relatively
low quadrature orders.
We assume that this is due to the algebraic interpolation
(\ref{eq:hybrid}) exactly reproducing the original matrix blocks as soon
as their rank is reached.

Figure~\ref{fig:greenhybrid_sub_time} illustrates that the hybrid method
is significantly faster than the pure quadrature method:
at comparable accuracies, the hybrid method saves more than 50\% of the
computation time, and the $\mathcal{H}^2$-matrix version saves approximately
50\% compare to the $\mathcal{H}$-matrix version.

%
%
\begin{figure}[Ht!]
 \begin{center}
  \subfigure[Memory per degree of freedom]{
            \includegraphics[width=0.47\textwidth]{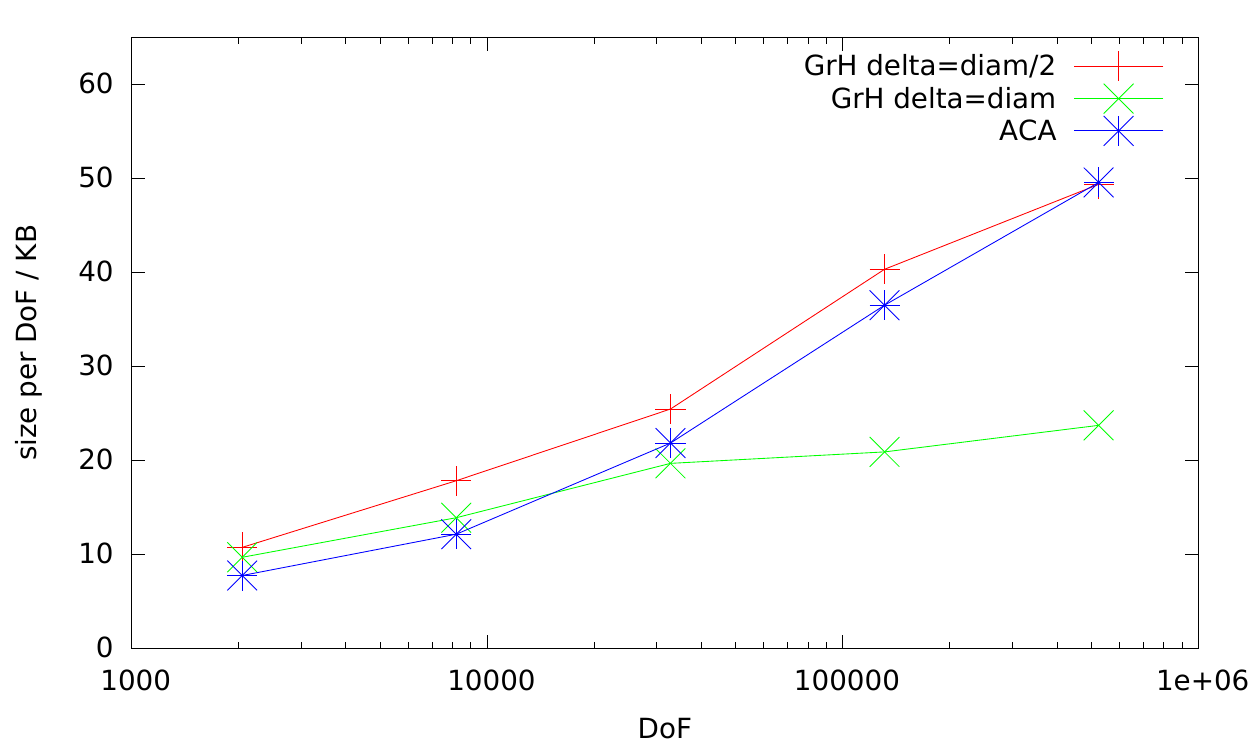}
        }
  \subfigure[Time per degree of freedom]{
            \includegraphics[width=0.47\textwidth]{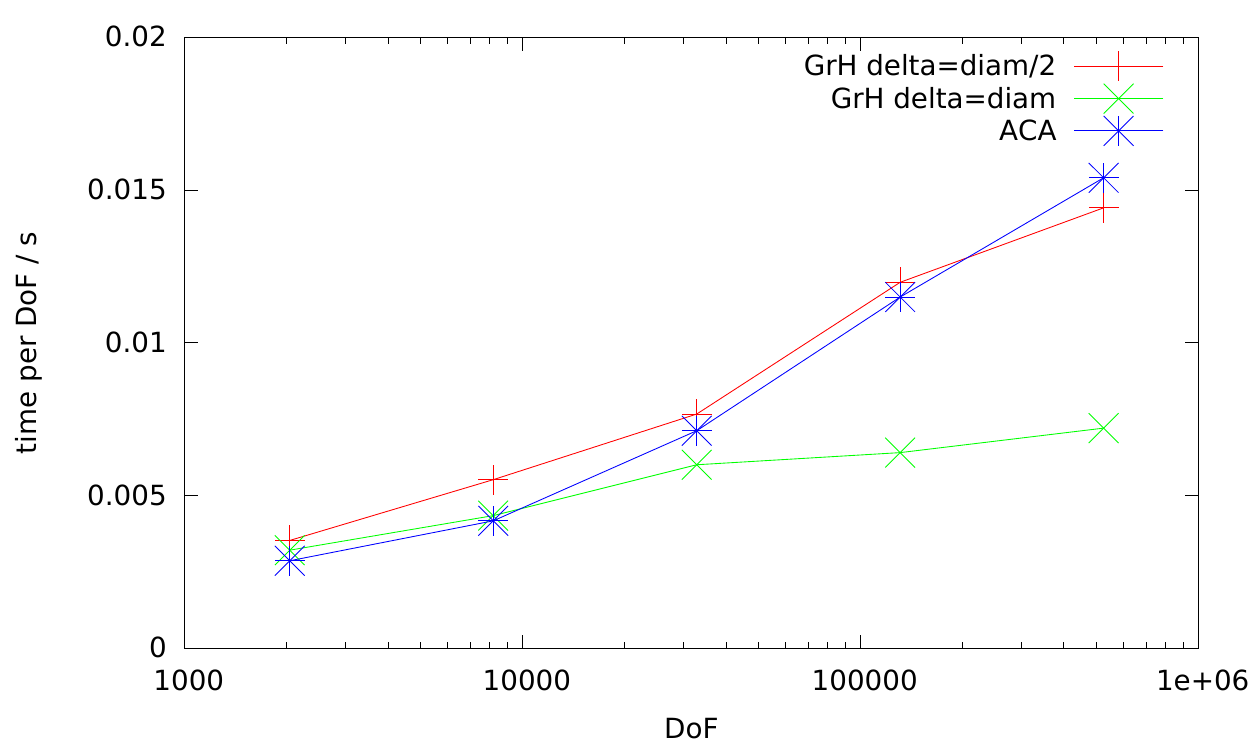}
        }
 \end{center}
 \caption{Unit sphere: memory consumption and setup time for SLP with the
          Green hybrid method,
          $\delta_t = \diam_\infty({\mathcal B}_t)/2$ and
          $\delta_t = \diam_\infty({\mathcal B}_t)$ compared to ACA.}
 \label{img:sphere_eta1_delta_05_10}
\end{figure}

\paragraph*{Choice of $\delta_t$.}\quad
In a second example, we apply the Green hybrid method to different
surface meshes on the unit sphere with the admissibility condition
(\ref{eq:admiss_grh_h2}) and the value
$\delta_t = \diam_\infty({\mathcal B}_t)/2$ used in the theoretical
investigation as well as the value
$\delta_t = \diam_\infty({\mathcal B}_t)$.
The latter is not covered by our theory, but
Figure~\ref{img:sphere_eta1_delta_05_10} shows that it is superior
both with regards to computational work and memory consumption.

Figure~\ref{img:sphere_eta1_delta_05_10} also shows the memory
and time requirements of the standard ACA technique.
We can see that the new method with $\delta_t = \diam_\infty(\mathcal{B}_t)$
offers significant advantages in both respects.

\begin{figure}[Ht!]
 \begin{center}
  \subfigure[Memory per degree of freedom]{
            \includegraphics[width=0.47\textwidth]{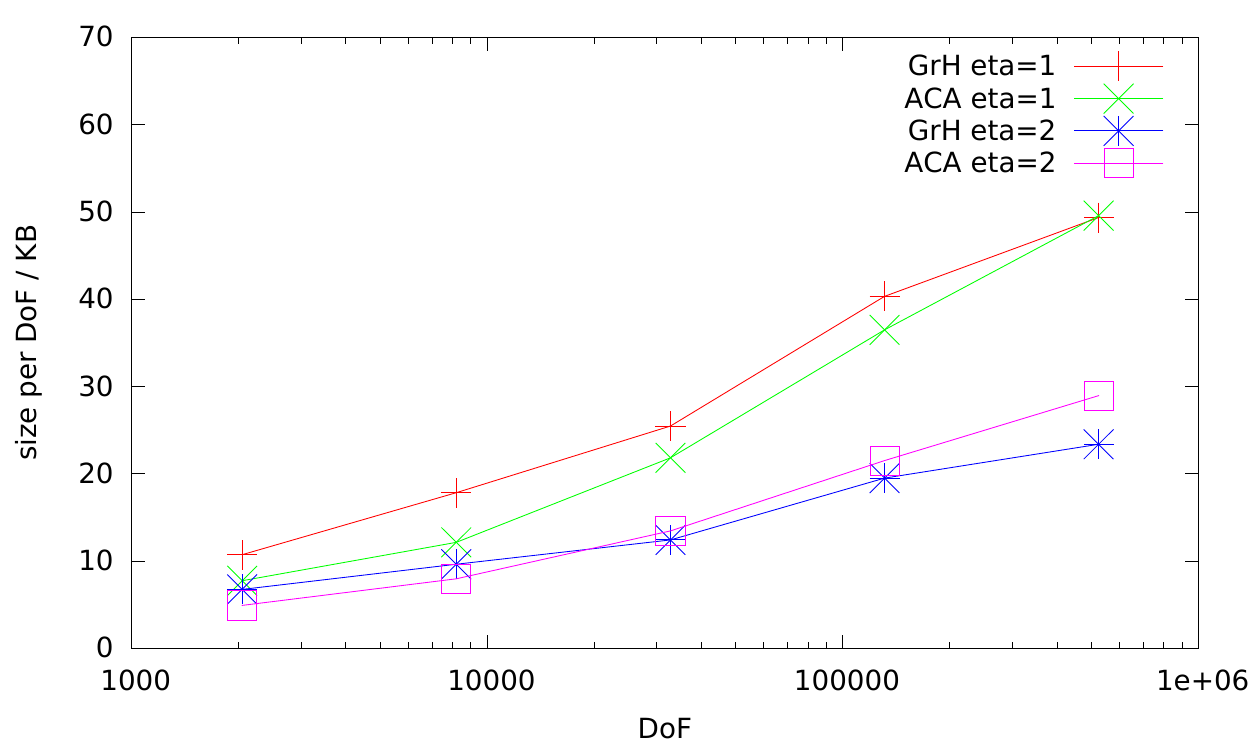}
        }
  \subfigure[Time per degree of freedom]{
            \includegraphics[width=0.47\textwidth]{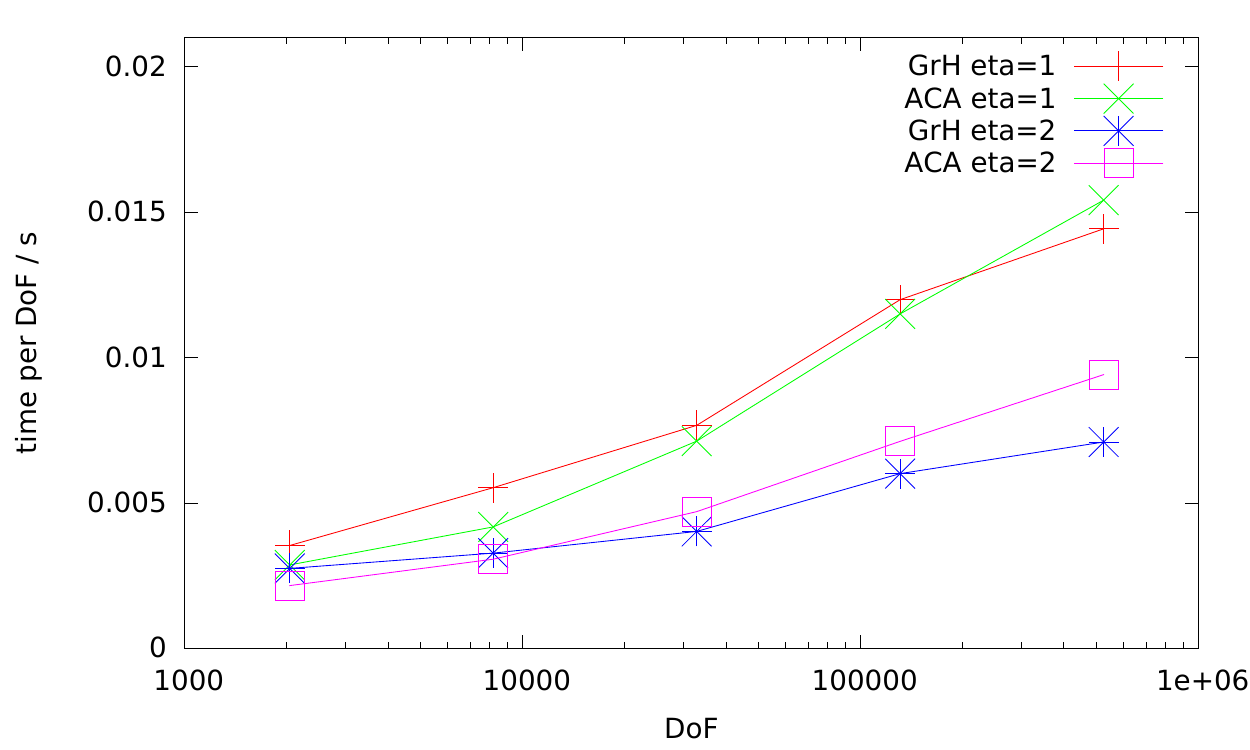}
        }
 \end{center}
 \caption{Unit sphere: memory consumption and setup time for SLP with the
          Green hybrid method and adaptive cross approximation.
          Results are shown with $\delta_t = \diam_{\infty}
          \left( \mathcal B_t \right) / 2$ for $\eta = 1$ and for $\eta = 2$.}
 \label{img:sphere_eta1_2}
\end{figure}

\paragraph*{Weaker admissibility condition.}\quad
Another useful modification that is not covered by our theory is the
construction of the block tree based on the admissibility condition
\begin{equation}
 \max \{ \diam_\infty({\mathcal B}_t), \diam_\infty({\mathcal B}_s) \}
 \leq \eta \dist_\infty({\mathcal B}_t, {\mathcal B}_s).
 \label{eq:weak_admiss}
\end{equation}
If we choose $\eta>1$, this condition is weaker than the condition
(\ref{eq:admiss_grh_h2}) used in the theoretical investigation.

Figure~\ref{img:sphere_eta1_2} shows experimental results for the
choices $\eta=1$ and $\eta=2$.
We can see that both ACA and the new Green hybrid method profit
from the weaker admissibility condition.

\paragraph*{Linear scaling.}\quad
In order to demonstrate the linear complexity of the Green hybrid method,
we have also computed the single layer potential for the same number of
degrees of freedom as before but using the same $m$ and $\epsilon_{ACA}$
for all resolutions of the sphere.
The results can be seen in Figure~\ref{img:linearscale}.

We can see that both the time and the storage requirements of the
new method indeed scale linearly with $n$.
Since the standard ACA algorithm \cite[Algorithm~4.2]{BERJ01} constructs
an $\mathcal{H}$-matrix instead of an $\mathcal{H}^2$-matrix, we observe
the expected $\mathcal{O}(n \log n)$ complexity.

\begin{figure}[Ht!]
 \begin{center}
  \subfigure[Memory per degree of freedom]{
            \includegraphics[width=0.47\textwidth]{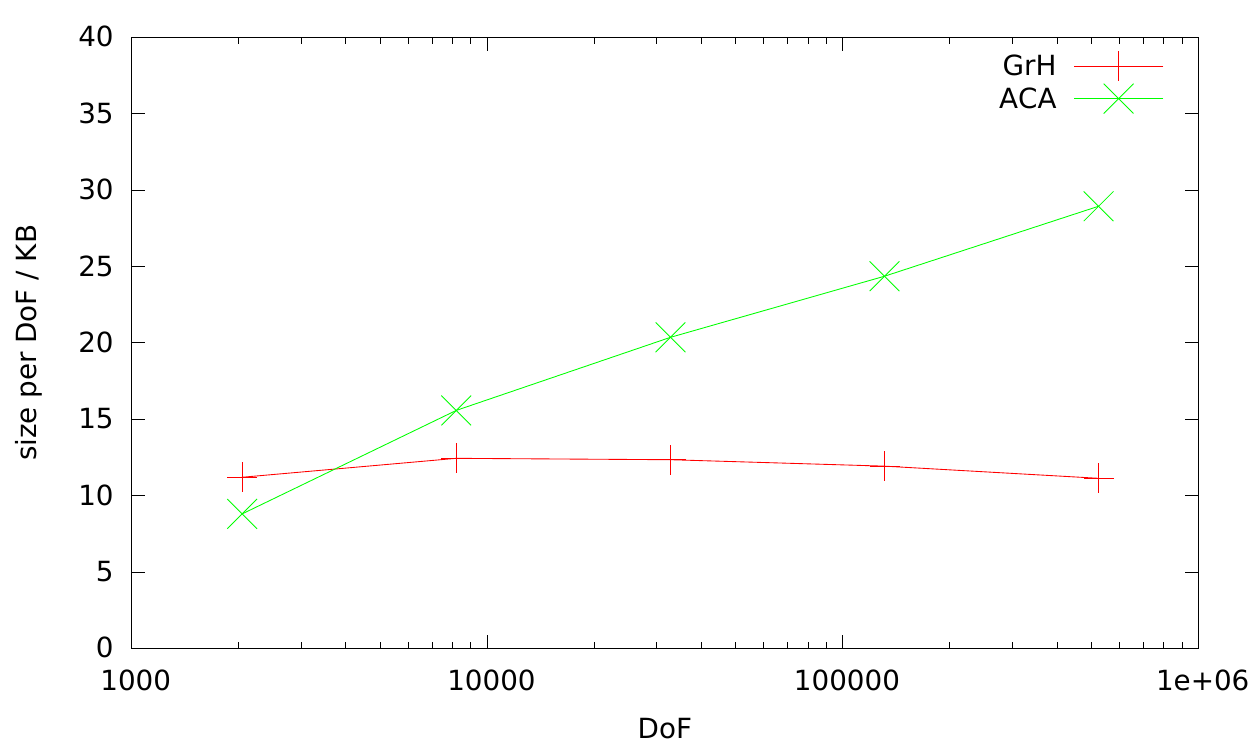}
        }
  \subfigure[Time per degree of freedom]{
            \includegraphics[width=0.47\textwidth]{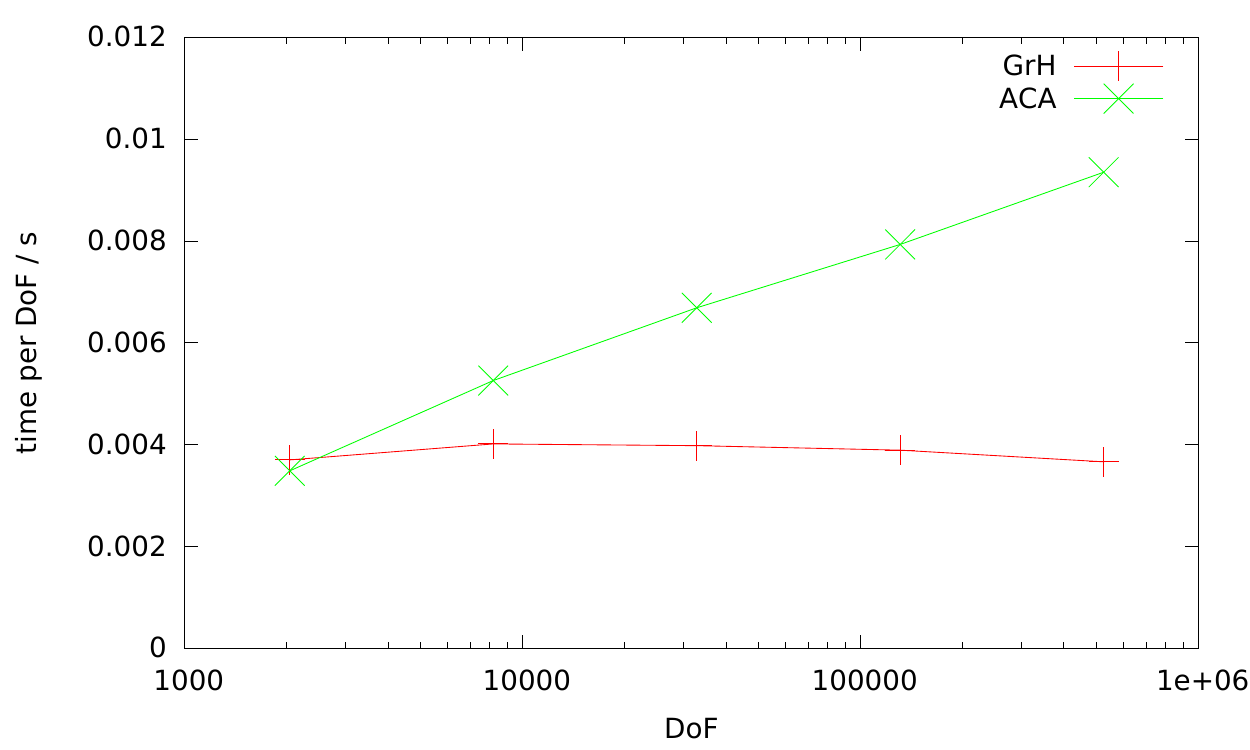}
        }
 \end{center}
 \caption{Unit sphere: memory consumption and setup time for SLP with fixed quadrature
          ranks, cross approximation error tolerances, minimal leafsizes, 
          $\eta = 2$ and $\delta_t = \diam_\infty({\mathcal B}_t)$. }
 \label{img:linearscale}
\end{figure}

\paragraph*{Combination with algebraic recompression.}\quad
We have seen that the Green hybrid method works fine with 
$\delta_t = \diam_\infty \left( \mathcal B_t \right) / 2$ as well as with 
$\delta_t = \diam_\infty \left( \mathcal B_t \right)$. 
It is also possible to use the weaker admissibility condition
\eqref{eq:weak_admiss}.
In order to obtain the best possible results the new method, we use
$\eta=2$, $\delta_t = \diam_\infty \left( \mathcal B_t \right)$ and apply
algebraic $\mathcal{H}^2$-recompression (cf. \cite[Section~6.6]{BO10})
to further reduce the storage requirements.
Table~\ref{tab:sphere_ghb_aca} shows the total time and size for different
resolutions of the unit sphere as well as the $L_2$-error observed for
our test functions.
We can see that the error converges at the expected rate for a piecewise
constant approximation, i.e., the matrix approximation is sufficiently
accurate.

For the ACA method, we also use algebraic recompression based on the
blockwise singular value decomposition and truncation.
Figure~\ref{img:sphere_ghb_aca} shows memory requirements and
compute times per degree of freedom for SLP and DLP matrices.
Since the nodal basis used by the DLP matrix requires three local basis
functions per triangle, the computational work is approximately three
times as high as for the piecewise constant basis.
Our implementation of ACA apparently reacts strongly to the higher
number of local basis functions required by the nodal basis, probably
due to the fact that ACA requires individual rows and columns and
cannot easily be optimized to take advantage of triangles shared among
different degrees of freedom.

\begin{table}[Ht!]
\begin{center}
\begin{tabular}{r|rr|rr|rr|rrr}
 &  &  & \multicolumn{2}{c}{SLP} & \multicolumn{2}{c}{DLP} & & & \\
$n$ & $m$ & $\epsilon_{ACA}$ & time & size  & time & size &$\epsilon_1$ &  $\epsilon_2$ &	$\epsilon_3$ \\
\hline
2048 & 2 & 5.0e-4 & 5 & 6 & 12 & 5 & 1.3e-1 & 2.4e-2 & 1.8e-1\\
8192 & 2 & 1.0e-4 & 23 & 30 & 65 & 25 & 6.3e-2 & 1.2e-2 & 9.0e-2\\
32768 & 2 & 1.0e-5 & 114 & 167 & 309 & 134 & 3.1e-2 & 5.6e-3 & 4.4e-2\\
131072 & 2 & 5.0e-6 & 470 & 751 & 1335 & 596 & 1.6e-2 & 2.9e-3 & 2.2e-2\\
524288 & 2 & 1.0e-6 & 2090 & 3692 & 6378 & 2936 & 7.8e-3 & 1.5e-3 & 1.1e-2
\end{tabular}
\caption{Unit sphere: Setup time in seconds, resulting size of SLP and DLP in megabytes
	 and absolute $L_2$-errors for different Dirichlet data using the 
	 new hybrid method. Order of quadrature for Green's formula is given
	 by $m$ and accuracy used by cross approximation and
         $\mathcal{H}^2$-recompression is given by $\epsilon_{ACA}$,
         $\delta_t = \diam_\infty(\mathcal{B}_t)$, $\eta = 2$ for both
         SLP and DLP.}
\label{tab:sphere_ghb_aca}
\end{center}
\end{table}

\begin{figure}[Ht!]
 \begin{center}
  \subfigure[Memory per degree of freedom]{
            \includegraphics[width=0.47\textwidth]{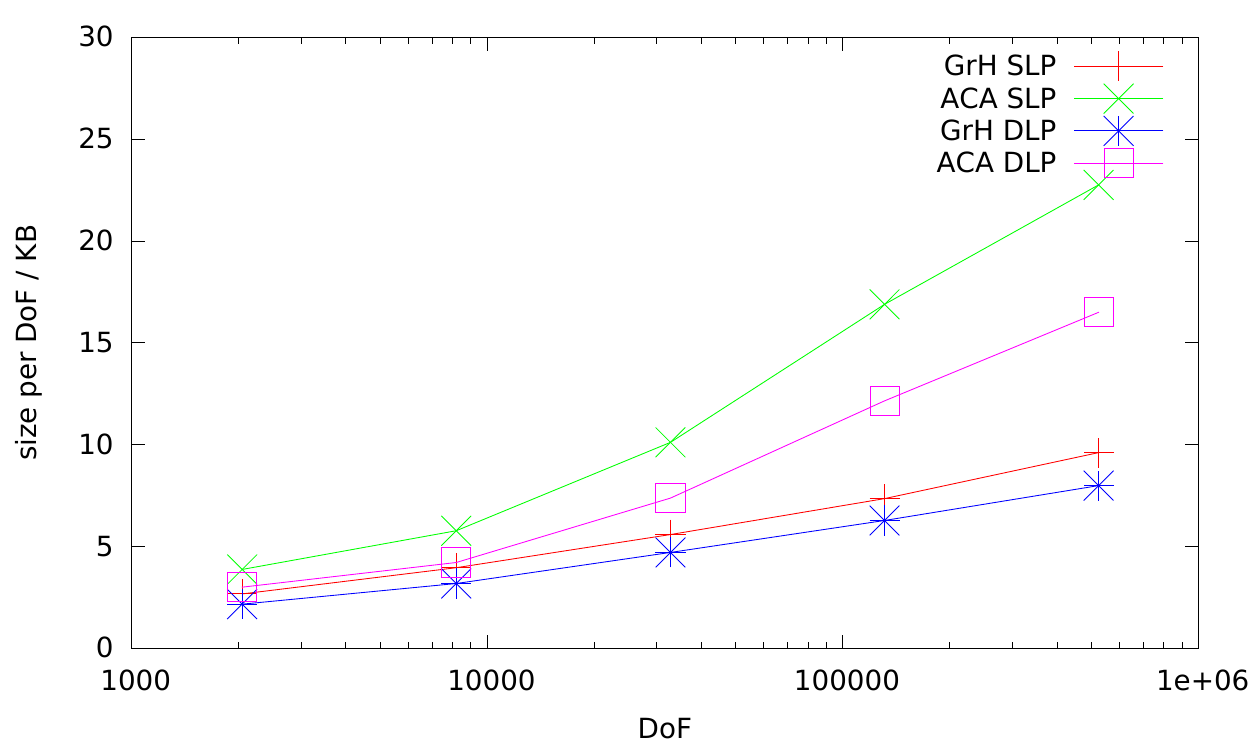}
        }
  \subfigure[Time per degree of freedom]{
            \includegraphics[width=0.47\textwidth]{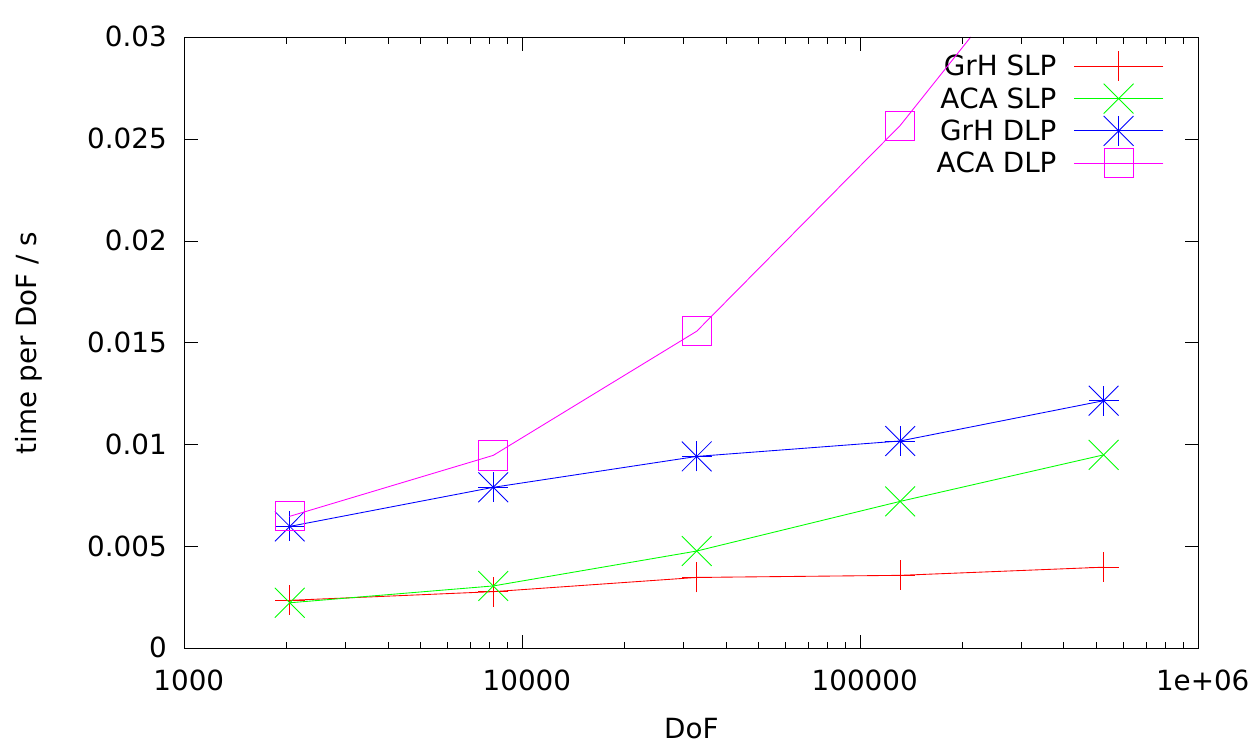}
        }
 \end{center}
 \caption{Unit sphere: memory consumption and setup time for SLP and DLP with 
	  $\delta_t = \diam_\infty \left( \mathcal B_t \right)$ and $\eta = 2$. In both cases
	  algebraic recompression techniques are included.}
 \label{img:sphere_ghb_aca}
\end{figure}

\paragraph*{Crankshaft.}\quad
Our new approximation technique is not only capable of handling
the simple unit sphere but also of more complex geometries such
as the well-known ``crankshaft'' geometry contained in the
the \emph{netgen} package of Joachim Sch\"oberl \cite{SC97}.
We have used the program to created meshes with 1748, 6992,
27968 and 111872 triangles.

In order to obtain $\mathcal{O}(h)$ convergence of the Neumann
data, we have to raise the nearfield quadrature order to 7
Gauss points per dimension for regular integrals and to 9
Gauss points for singular ones.
The results are shown in Table~\ref{tab:shaft_greenhybrid}.

Due to the significantly higher number of nearfield quadrature
points, the total computing time is far higher than for the
simple unit sphere.
We also have to choose $m=3$ for the Green quadrature method
and significantly lower error tolerances for the hybrid method
and the recompression in order to recover the discretization
error.
Except for these changes, the new method works as expected.

\begin{table}[Ht!]
\begin{center}
\begin{tabular}{r|rr|rr|rr|rrr}
 &  &  & \multicolumn{2}{c}{SLP} & \multicolumn{2}{c}{DLP} & & & \\
Dof  &   $m$ & $\epsilon_{ACA}$ & time & size  & time & size &$\epsilon_1$ &  $\epsilon_2$ & $\epsilon_3$ \\
\hline
1748 & 2 & 1.0e-5 & 149 & 13 & 346 & 9 & 8.8e-2 & 9.5e-3 & 3.6e-2\\
6992 & 3 & 1.0e-6 & 1471 & 107 & 3422 & 90 & 3.1e-2 & 4.2e-3 & 1.6e-2\\
27968 & 3 & 1.0e-7 & 10660 & 633 & 27410 & 722 & 1.2e-2 & 2.1e-3 & 7.7e-3\\
111872 & 3 & 1.0e-8 & 59790 & 2985 & 182200 & 3829 & 4.7e-3 & 1.2e-3 & 3.9e-3
\end{tabular}
\caption{Crankshaft: setup time in seconds, resulting size of SLP and DLP in megabytes
    and absolute $L^2$-errors for different Dirichlet data using the 
    new hybrid method. Order of quadrature for Green's formula is given
    by $m$ and accuracy used by ACA and $\mathcal H^2$-recompression is 
    given by $\epsilon_{ACA}$, $\delta_t = \diam_{\infty} 
    \left( \mathcal B_t \right)$, $\eta = 1$ for both SLP and DLP.}
\label{tab:shaft_greenhybrid}
\end{center}
\end{table}

\appendix
\section{Proofs of technical lemmas}
\label{se:appendix}

\emph{Proof of Lemma~\ref{le:omega_t}:}
By definition of the maximum norm, we can find $\iota\in\{1,\ldots,d\}$
such that $2\delta_t = \diam_\infty({\mathcal B}_t) =
b_{t,\iota}-a_{t,\iota}$ and $b_{t,\kappa}-a_{t,\kappa}\leq 2\delta_t$ holds
for all $\kappa\in\{1,\ldots,d\}$.
Most of our claims are direct consequences of this estimate, only the
last claim of (\ref{eq:diam_dist}) requires a closer look.
Let $x\in\partial\omega_t$ and $y\in{\mathcal F}_t$ be given with
\begin{equation*}
  \|x-y\|_\infty = \dist_\infty(\partial\omega_t, {\mathcal F}_t).
\end{equation*}
Let $\widehat x\in {\mathcal B}_t$ be a point in ${\mathcal B}_t$ that
has minimal distance to $x$.
By construction, we have
\begin{equation*}
  \|x-\widehat x\|_\infty \leq \delta_t,
\end{equation*}
and the triangle inequality in combination with (\ref{eq:farfield})
yields
\begin{align*}
  \dist_\infty(\partial\omega_t, {\mathcal F}_t)
  &= \|x-y\|_\infty
   \geq \|\widehat x-y\|_\infty - \|\widehat x-x\|_\infty
   \geq \dist_\infty({\mathcal B}_t, B_s) - \delta_t\\
  &\geq \diam_\infty({\mathcal B}_t) - \diam_\infty({\mathcal B}_t)/2
   = \diam_\infty({\mathcal B}_t)/2 = \delta_t,
\end{align*}
and this is the required estimate.
\qed

\medskip

\noindent
\emph{Proof of Lemma~\ref{le:derivatives}:}
Let $\kappa\in\bbbn_0^d$ be a multiindex.
We consider the function
\begin{align*}
  g_x : [-1,1]^d &\to \bbbr, &
        \hat z &\mapsto (\partial^\kappa g)(x,\Phi_t(\hat z))
\end{align*}
and aim to prove
\begin{align}\label{eq:gx_derivative}
  \partial^\nu g_x(\hat z)
  &= s^\nu (\partial^{\nu+\kappa} g)(x,\Phi_t(\hat z)) &
  &\text{ for all } \hat z\in [-1,1]^d,\ \nu\in\bbbn_0^d
\end{align}
with the vector
\begin{equation*}
  s := \begin{pmatrix}
    (b_{t,1}-a_{t,1}+2\delta_t)/2\\
    \vdots\\
    (b_{t,d}-a_{t,d}+2\delta_t)/2
  \end{pmatrix} \in\bbbr^d.
\end{equation*}
We proceed by induction:
for the multiindex $\nu=0$, the identity (\ref{eq:gx_derivative})
is trivial.

Let $m\in\bbbn_0$ and assume that (\ref{eq:gx_derivative}) has been
proven for all multiindices $\nu\in\bbbn_0^d$ with $|\nu|\leq m$.
Let $\nu\in\bbbn_0^d$ be a multiindex with $|\nu|=m+1$.
Then we can find $\mu\in\bbbn_0^d$ with $|\mu|=m$ and $i\in\{1,\ldots,d\}$
such that $\nu=(\mu_1,\ldots,\mu_{i-1},\mu_i+1,\mu_{i+1},\ldots,\mu_d)$.
This implies
\begin{align*}
  \partial^\nu g_x(\hat z)
  &= \frac{\partial}{\partial \hat z_i} (\partial^{\mu+\kappa} g_x)(\hat z) &
  &\text{ for all } \hat z\in [-1,1]^d.
\end{align*}
Applying the induction assumption and the chain rule yields
\begin{align*}
  \partial^\nu g_x(\hat z)
  &= \frac{\partial}{\partial \hat z_i} \partial^\mu g_x(\hat z)
   = \frac{\partial}{\partial \hat z_i} s^\mu
           (\partial^{\mu+\kappa} g)(x,\Phi_t(\hat z))\\
  &= s^\mu \frac{b_{t,i}-a_{t,i}+2\delta_t}{2}
           \left(\frac{\partial}{\partial \hat z_i} \partial^{\mu+\kappa} g\right)
             (x,\Phi_t(\hat z))\\
  &= s^\nu (\partial^{\nu+\kappa} g)(x,\Phi_t(\hat z))
       \qquad\text{ for all } \hat z\in[-1,1]^d
\end{align*}
since $D\Phi_t$ is a diagonal matrix due to (\ref{eq:Phi_t}).
The induction is complete.

By definition (\ref{eq:parametrization}), we have
\begin{align*}
  \gamma_\iota(\hat z)
  &= \Phi_t(\hat z_1,\ldots,\hat z_{\lceil \iota/2 \rceil-1}, \pm 1,
            \hat z_{\lceil \iota/2 \rceil}, \ldots, \hat z_{d-1}),
\end{align*}
therefore (\ref{eq:gx_derivative}) implies
\begin{subequations}\label{eq:gamma_derivatives}
\begin{align}
  \partial^{\hat\nu} \hat f_1(\hat z)
  &= s^\nu (\partial_y^\nu g)(x,\gamma_\iota(\hat z)) &
  &\text{ for all } \hat z\in Q,\ \hat\nu\in\bbbn_0^{d-1}
\end{align}
with
\begin{equation*}
  \nu := (\hat\nu_1, \ldots, \hat\nu_{\lceil \iota/2\rceil-1}, 0,
          \hat\nu_{\lceil \iota/2\rceil}, \ldots, \hat\nu_{d-1}).
\end{equation*}
The exterior normal vector on the surface $\gamma_\iota(Q)$ is
the $\iota/2$-th canonical unit vector if $\iota$ is even and
the negative $(\iota+1)/2$-th canonical unit vector if it is uneven,
so we obtain also
\begin{align}
  \partial^{\hat\nu} \hat g_2(\hat z)
  &= \pm s^\nu (\partial_y^{\nu+\kappa} g)(x,\gamma_\iota(\hat z)) &
  &\text{ for all } \hat z\in Q,\ \hat\nu\in\bbbn_0^{d-1},
\end{align}
where $\kappa$ is the $\lceil \iota/2 \rceil$-th canonical unit
vector in $\bbbn_0^d$.

Exchanging the roles of $x$ and $y$ in these arguments yields
\begin{align}
  \partial^{\hat\nu} \hat f_2(\hat z)
  &= s^\nu (\partial_x^\nu g)(\gamma_\iota(\hat z),y),\\
  \partial^{\hat\nu} \hat g_1(\hat z)
  &= \pm s^\nu (\partial_x^{\nu+\kappa} g)(\gamma_\iota(\hat z),y) &
  &\text{ for all } \hat z\in Q,\ \hat\nu\in\bbbn_0^{d-1}.
\end{align}
\end{subequations}
Now we only have to combine the equations (\ref{eq:gamma_derivatives})
with
\begin{align*}
  |s^\nu| &\leq (2\delta_t)^{|\nu|} &
  &\text{ for all } \nu\in\bbbn_0^d
\end{align*}
and the asymptotic smoothness (\ref{eq:asymptotically_smooth}) to
obtain
\begin{align*}
  |\partial^{\hat\nu} \hat f_1(\hat z)|
  &\leq (2\delta_t)^{|\hat\nu|} C_{\rm as} \hat\nu!
        \frac{c_0^{|\hat\nu|}}
             {\|x-\gamma_\iota(\hat z)\|^{\sigma+|\hat\nu|}}\\
  &\leq \frac{C_{\rm as} \hat\nu!}{\delta_t^\sigma}
        \left( \frac{2 \delta_t c_0}{\delta_t} \right)^{|\hat\nu|}
   = \frac{C_{\rm as} \hat\nu!}{\delta_t^\sigma} (2 c_0)^{|\hat\nu|},\\
  |\partial^{\hat\nu} \hat f_2(\hat z)|
  &\leq (2\delta_t)^{|\hat\nu|} C_{\rm as} \hat\nu!
        \frac{c_0^{|\hat\nu|}}
             {\|\gamma_\iota(\hat z)-y\|^{\sigma+|\hat\nu|}}\\
  &\leq \frac{C_{\rm as} \hat\nu!}{\delta_t^\sigma}
        \left( \frac{2 \delta_t c_0}{\delta_t} \right)^{|\hat\nu|}
   = \frac{C_{\rm as} \hat\nu!}{\delta_t^\sigma} (2 c_0)^{|\hat\nu|},\\
  |\partial^{\hat\nu} \hat g_1(\hat z)|
  &\leq (2\delta_t)^{|\hat\nu|} C_{\rm as} \hat\nu!
        \frac{c_0^{|\hat\nu|+1}}
             {\|x-\gamma_\iota(\hat z)\|^{\sigma+1+|\hat\nu|}}\\
  &\leq \frac{C_{\rm as} c_0 \hat\nu!}{\delta_t^{\sigma+1}}
        \left( \frac{2 \delta_t c_0}{\delta_t} \right)^{|\hat\nu|}
   = \frac{C_{\rm as} \hat\nu!}{\delta_t^{\sigma+1}} (2 c_0)^{|\hat\nu|},\\
  |\partial^{\hat\nu} \hat g_2(\hat z)|
  &\leq (2\delta_t)^{|\hat\nu|} C_{\rm as} \hat\nu!
        \frac{c_0^{|\hat\nu|+1}}
             {\|\gamma_\iota(\hat z)-y\|^{\sigma+1+|\hat\nu|}}\\
  &\leq \frac{C_{\rm as} c_0 \hat\nu!}{\delta_t^{\sigma+1}}
        \left( \frac{2 \delta_t c_0}{\delta_t} \right)^{|\hat\nu|}
   = \frac{C_{\rm as} \hat\nu!}{\delta_t^{\sigma+1}} (2 c_0)^{|\hat\nu|},
\end{align*}
where Lemma~\ref{le:omega_t} provides the lower bounds
$\delta_t\leq\|x-\gamma_\iota(\hat z)\|$ and
$\delta_t\leq\|\gamma_\iota(\hat z)-y\|$ and we have taken advantage
of $(\nu+\kappa)!=\nu!=\hat\nu!$ due to $\nu_{\lceil\iota/2\rceil}=0$.
\qed

\bibliographystyle{plain}
\bibliography{scicomp}

\begin{thebibliography}{10}

\bibitem{AN92}
C.~R. Anderson.
\newblock An implementation of the fast multipole method without multipoles.
\newblock {\em SIAM J. Sci. Stat. Comp.}, 13:923--947, 1992.

\bibitem{BE00a}
M.~Bebendorf.
\newblock Approximation of boundary element matrices.
\newblock {\em {N}umer. {M}ath.}, 86(4):565--589, 2000.

\bibitem{BEGR06}
M.~Bebendorf and R.~Grzhibovskis.
\newblock Accelerating {G}alerkin {BEM} for linear elasticity using adaptive
  cross approximation.
\newblock {\em Math. Meth. Appl. Sci.}, 29:1721--1747, 2006.

\bibitem{BERJ01}
M.~Bebendorf and S.~Rjasanow.
\newblock {A}daptive {L}ow-{R}ank {A}pproximation of {C}ollocation {M}atrices.
\newblock {\em Computing}, 70(1):1--24, 2003.

\bibitem{BEVE12}
M.~Bebendorf and R.~Venn.
\newblock Constructing nested bases approximations from the entries of
  non-local operators.
\newblock {\em Num. Math.}, 121(4):609--635, 2012.

\bibitem{BO10}
S.~B{\"o}rm.
\newblock {\em Efficient Numerical Methods for Non-local Operators: {${\mathcal
  H}^2$}-Matrix Compression, Algorithms and Analysis}, volume~14 of {\em EMS
  Tracts in Mathematics}.
\newblock EMS, 2010.

\bibitem{BOGO12}
S.~{B\"orm} and J.~{G\"ordes}.
\newblock Low-rank approximation of integral operators by using the {Green}
  formula and quadrature.
\newblock {\em Numerical Algorithms}, 64(3):567--592, 2013.

\bibitem{BOGR04}
S.~B{\"o}rm and L.~Grasedyck.
\newblock Hybrid cross approximation of integral operators.
\newblock {\em Numer. Math.}, 101:221--249, 2005.

\bibitem{BOGRHA03a}
S.~B{\"o}rm, L.~Grasedyck, and W.~Hackbusch.
\newblock {H}ierarchical {M}atrices.
\newblock Lecture Note 21 of the Max Planck Institute for Mathematics in the
  Sciences, 2003.

\bibitem{BOHA02}
S.~B{\"o}rm and W.~Hackbusch.
\newblock Data-sparse approximation by adaptive {${\mathcal{H}}^2$}-matrices.
\newblock {\em Computing}, 69:1--35, 2002.

\bibitem{BOHA02a}
S.~B{\"o}rm and W.~Hackbusch.
\newblock {${\mathcal{H}}^2$}-matrix approximation of integral operators by
  interpolation.
\newblock {\em Appl. Numer. Math.}, 43:129--143, 2002.

\bibitem{BR91}
A.~Brandt.
\newblock Multilevel computations of integral transforms and particle
  interactions with oscillatory kernels.
\newblock {\em Comput. Phys. Comm.}, 65(1--3):24--38, 1991.

\bibitem{BRLU90}
A.~Brandt and A.~A. Lubrecht.
\newblock Multilevel matrix multiplication and fast solution of integral
  equations.
\newblock {\em J. Comput. Phys.}, 90:348--370, 1990.

\bibitem{DAPRSC94b}
W.~Dahmen, S.~Pr{\"o}ssdorf, and R.~Schneider.
\newblock Wavelet approximation methods for pseudodifferential equations {I}:
  {S}tability and convergence.
\newblock {\em Math. Z.}, 215:583--620, 1994.

\bibitem{DASC99}
W.~Dahmen and R.~Schneider.
\newblock Wavelets on manifolds {I}: {C}onstruction and domain decomposition.
\newblock {\em SIAM J. Math. Anal.}, 31:184--230, 1999.

\bibitem{GI01}
K.~Giebermann.
\newblock {M}ultilevel approximation of boundary integral operators.
\newblock {\em Computing}, 67:183--207, 2001.

\bibitem{GOTYZA97}
S.~A. Goreinov, E.~E. Tyrtyshnikov, and N.~L. Zamarashkin.
\newblock A theory of pseudoskeleton approximations.
\newblock {\em Lin. Alg. Appl.}, 261:1--22, 1997.

\bibitem{GRHA02}
L.~Grasedyck and W.~Hackbusch.
\newblock Construction and arithmetics of {${\mathcal{H}}$}-matrices.
\newblock {\em Computing}, 70:295--334, 2003.

\bibitem{GRRO87}
L.~Greengard and V.~Rokhlin.
\newblock A fast algorithm for particle simulations.
\newblock {\em J. Comp. Phys.}, 73:325--348, 1987.

\bibitem{HA92}
W.~Hackbusch.
\newblock {\em {E}lliptic {D}ifferential {E}quations. {T}heory and {N}umerical
  {T}reatment}.
\newblock Springer-Verlag Berlin, 1992.

\bibitem{HA99}
W.~Hackbusch.
\newblock A sparse matrix arithmetic based on $\mathcal{H}$-matrices. {P}art
  {I}: {I}ntroduction to $\mathcal{H}$-matrices.
\newblock {\em Computing}, 62:89--108, 1999.

\bibitem{HA09}
W.~Hackbusch.
\newblock {\em {H}ierarchische {M}atrizen --- {A}lgorithmen und {A}nalysis}.
\newblock Springer, 2009.

\bibitem{HAKH00}
W.~Hackbusch and B.~N. Khoromskij.
\newblock A sparse matrix arithmetic based on $\mathcal{H}$-matrices. {P}art
  {II}: {A}pplication to multi-dimensional problems.
\newblock {\em Computing}, 64:21--47, 2000.

\bibitem{HANO89}
W.~Hackbusch and Z.~P. Nowak.
\newblock On the fast matrix multiplication in the boundary element method by
  panel clustering.
\newblock {\em Numer. Math.}, 54:463--491, 1989.

\bibitem{HASC06}
H.~Harbrecht and R.~Schneider.
\newblock Wavelet {G}alerkin schemes for boundary integral equations --
  {I}mplementation and quadrature.
\newblock {\em SIAM J. Sci. Comput.}, 27:1347--1370, 2006.

\bibitem{HSWE08}
G.~C. Hsiao and W.~L. Wendland.
\newblock {\em Boundary Integral Equations}.
\newblock Number 164 in Appl. Math. Sci. Springer, 2008.

\bibitem{MAMITI08}
R.~Maaskant, R.~Mittra, and A.~Tijhuis.
\newblock Fast analysis of large antenna arrays using the characteristic basis
  function method and the adaptive cross approximation algorithm.
\newblock {\em IEEE Trans. Ant. Prop.}, 56(11):3440--3451, 2008.

\bibitem{RO85}
V.~Rokhlin.
\newblock Rapid solution of integral equations of classical potential theory.
\newblock {\em J. Comp. Phys.}, 60:187--207, 1985.

\bibitem{SA96}
S.~A. Sauter.
\newblock Cubature techniques for 3-d {G}alerkin {BEM}.
\newblock In W.~Hackbusch and G.~Wittum, editors, {\em Boundary Elements:
  Implementation and Analysis of Advanced Algorithms}, pages 29--44.
  Vieweg-Verlag, 1996.

\bibitem{SASC04}
S.~A. Sauter and C.~Schwab.
\newblock {\em Randelementmethoden}.
\newblock Teubner, 2004.

\bibitem{SC97}
J.~{Sch\"oberl}.
\newblock {NETGEN} --- {An} advancing front {2D}/{3D}-mesh generator based on
  abstract rules.
\newblock {\em Comp. Vis. Sci.}, 1(1):41--52, 1997.

\bibitem{TAHERI11}
J.~M. Tamayo, A.~Heldring, and J.~M. Rius.
\newblock Multilevel adaptive cross approximation.
\newblock {\em IEEE Trans. Ant. Prop.}, 59(12):4600--4608, 2011.

\bibitem{TY96}
E.~E. Tyrtyshnikov.
\newblock Mosaic-skeleton approximation.
\newblock {\em {C}alcolo}, 33:47--57, 1996.

\bibitem{TY99}
E.~E. Tyrtyshnikov.
\newblock Incomplete cross approximation in the mosaic-skeleton method.
\newblock {\em Computing}, 64:367--380, 2000.

\bibitem{BIYIZO04}
L.~Ying, G.~Biros, and D.~Zorin.
\newblock A kernel-independent adaptive fast multipole algorithm in two and
  three dimensions.
\newblock {\em J. Comp. Phys.}, 196(2):591--626, 2004.

\end{thebibliography}

\end{document}